\documentclass[11pt]{amsart}
\usepackage[arrow,matrix]{xy}
\usepackage{graphicx}
\usepackage{amsmath}
\usepackage{hyperref}
\usepackage{comment}
\usepackage{color}

\usepackage{color}
\newcommand{\blue}{\textcolor{blue}}
\newcommand{\red}{\textcolor{red}}
\newcommand{\green}{\textcolor{green}}

\numberwithin{equation}{section}
\theoremstyle{plain}
\newtheorem{lemma}{Lemma}[section]
\newtheorem{proposition}[lemma]{Proposition}
\newtheorem{theorem}[lemma]{Theorem}
\newtheorem{corollary}[lemma]{Corollary}

\theoremstyle{definition}
\newtheorem{definition}[lemma]{Definition}
\newtheorem{remark}[lemma]{Remark}
\newtheorem{example}[lemma]{Example}

\begin{document}
\newcommand{\R}{{\mathbb R}}
\newcommand{\C}{{\mathbb C}}
\newcommand{\F}{{\mathbb F}}
\renewcommand{\O}{{\mathbb O}}
\newcommand{\Z}{{\mathbb Z}} 
\newcommand{\N}{{\mathbb N}}
\newcommand{\Q}{{\mathbb Q}}
\renewcommand{\H}{{\mathbb H}}

\newcommand{\dass}{da\ss~}

\newcommand{\Aa}{{\mathcal A}}
\newcommand{\Bb}{{\mathcal B}}
\newcommand{\Cc}{{\mathcal C}}    
\newcommand{\Dd}{{\mathcal D}}
\newcommand{\Ee}{{\mathcal E}}
\newcommand{\Ff}{{\mathcal F}}
\newcommand{\Gg}{{\mathcal G}}    
\newcommand{\Hh}{{\mathcal H}}
\newcommand{\Kk}{{\mathcal K}}
\newcommand{\Ii}{{\mathcal I}}
\newcommand{\Jj}{{\mathcal J}}
\newcommand{\Ll}{{\mathcal L}}    
\newcommand{\Mm}{{\mathcal M}}    
\newcommand{\Nn}{{\mathcal N}}
\newcommand{\Oo}{{\mathcal O}}
\newcommand{\Pp}{{\mathcal P}}
\newcommand{\Qq}{{\mathcal Q}}
\newcommand{\Rr}{{\mathcal R}}
\newcommand{\Ss}{{\mathcal S}}
\newcommand{\Tt}{{\mathcal T}}
\newcommand{\Uu}{{\mathcal U}}
\newcommand{\Vv}{{\mathcal V}}
\newcommand{\Ww}{{\mathcal W}}
\newcommand{\Xx}{{\mathcal X}}
\newcommand{\Yy}{{\mathcal Y}}
\newcommand{\Zz}{{\mathcal Z}}

\newcommand{\zt}{{\tilde z}}
\newcommand{\xt}{{\tilde x}}
\newcommand{\Ht}{\widetilde{H}}
\newcommand{\ut}{{\tilde u}}
\newcommand{\Mt}{{\widetilde M}}
\newcommand{\Llt}{{\widetilde{\mathcal L}}}
\newcommand{\yt}{{\tilde y}}
\newcommand{\vt}{{\tilde v}}
\newcommand{\Ppt}{{\widetilde{\mathcal P}}}
\newcommand{\bp }{{\bar \partial}} 

\newcommand{\Remark}{{\it Remark}}
\newcommand{\Proof}{{\it Proof}}
\newcommand{\ad}{{\rm ad}}
\newcommand{\Om}{{\Omega}}
\newcommand{\om}{{\omega}}
\newcommand{\eps}{{\varepsilon}}
\newcommand{\Di}{{\rm Diff}}
\newcommand{\vol}{{\rm vol}}
\newcommand{\Pro}[1]{\noindent {\bf Proposition #1}}
\newcommand{\Thm}[1]{\noindent {\bf Theorem #1}}
\newcommand{\Lem}[1]{\noindent {\bf Lemma #1 }}
\newcommand{\An}[1]{\noindent {\bf Anmerkung #1}}
\newcommand{\Kor}[1]{\noindent {\bf Korollar #1}}
\newcommand{\Satz}[1]{\noindent {\bf Satz #1}}

\newcommand{\Cinf}{C^{\infty}}
\newcommand{\la}{\langle}
\newcommand{\ra}{\rangle}
\newcommand{\half}{\scriptstyle\frac{1}{2}}
\newcommand{\p}{{\partial}}
\newcommand{\notsub}{\not\subset}
\newcommand{\iI}{{I}}               
\newcommand{\bI}{{\partial I}}      
\newcommand{\LRA}{\Longrightarrow}
\newcommand{\LLA}{\Longleftarrow}
\newcommand{\lra}{\longrightarrow}
\newcommand{\LLR}{\Longleftrightarrow}
\newcommand{\lla}{\longleftarrow}
\newcommand{\INTO}{\hookrightarrow}

\newcommand{\QED}{\hfill$\Box$\medskip}
\newcommand{\UuU}{\Upsilon _{\delta}(H_0) \times \Uu _{\delta} (J_0)}
\newcommand{\bm}{\boldmath}

\newcommand{\commentgreen}[1]{\green{(*)}\marginpar{\fbox{\parbox[l]{3cm}{\green{#1}}}}}
\newcommand{\commentred}[1]{\red{(*)}\marginpar{\fbox{\parbox[l]{3cm}{\red{#1}}}}}
\newcommand{\commentblue}[1]{\blue{(*)}\marginpar{\fbox{\parbox[l]{3cm}{\blue{#1}}}}}

\title[Information geometry and sufficient statistics]{\large Information geometry and sufficient statistics}
\author[N. Ay, J. Jost, H.V. L\^e and L. Schwachh\"ofer]{Nihat Ay${}^{1,5}$, J\"urgen Jost${}^{1,4,5}$, H\^ong V\^an L\^e${}^{2}$  and Lorenz Schwachh\"ofer${}^{3}$}
 \date{\today}
 \thanks {J.J. is partially supported by ERC Advanced Grant FP7-267087; H.V.L. is partially supported by RVO: 67985840}

\medskip
\address{${}^{1}$Max-Planck-Institut f\"ur Mathematik in den Naturwissenschaften, Inselstrasse 22, 04103 Leipzig, Germany}
\address{${}^{2}$Institute  of Mathematics of ASCR,
Zitna 25, 11567  Praha 1, Czech Republic} 
\address{${}^{3}$Fakult\"at f\"ur Mathematik,
Technische Universit\"at Dortmund,
Vogelpothsweg 87, 44221 Dortmund, Germany} 
\address{${}^{4}$Mathematisches Institut,
Universit\"at Leipzig,
04081 Leipzig, Germany} 
\address{${}^{5}$Santa Fe Institute, 1399 Hyde Park Road, Santa Fe, NM 87501, USA}

\begin{abstract} Information geometry
  provides a geometric approach to families of statistical models. The
  key geometric structures are the Fisher quadratic form and the
  Amari-Chentsov tensor. In
  statistics, the notion of sufficient statistic expresses the
  criterion for passing from one model to another without loss of
  information. This leads to the question how the geometric structures
  behave under such sufficient statistics. While this is well studied
  in the finite sample size case, in the infinite case, we encounter
  technical problems concerning the appropriate topologies. Here, we
  introduce notions of parametrized measure models and tensor
  fields on them that exhibit the right behavior under statistical transformations. Within this framework, we can then handle the
  topological issues and show that the Fisher metric  and  the
Amari-Chentsov tensor  on statistical models in the class of symmetric
2-tensor fields and 3-tensor fields can be uniquely (up to a constant)
characterized by their invariance under
sufficient statistics, thereby achieving a full generalization of the
original result of Chentsov to infinite sample sizes. More generally, we  decompose Markov morphisms between
statistical models in terms of statistics. In particular,   a monotonicity result for the
Fisher information naturally follows.
\end{abstract}
\maketitle

 
\tableofcontents



\section{Introduction}
Let us begin with a short synopsis of our work and its context. Parametrized statistics deals with families of probability
  measures on some sample space $\Omega$ parametrized by a parameter
  $x$ from some space $M$ which we shall take to be  a Banach manifold, in particular,  a finite
  dimensional manifold. This
  parameter is to be estimated, and for that purpose, one wishes to
  quantify the dependence of the model on that parameter. That
  is achieved by the Fisher metric as first suggested by
  Rao \cite{Rao1945},  followed by  Jeffreys \cite{Jeffreys1946}, Efron \cite{Efron1975} and  then systematically developed by
  Chentsov and Morozova \cite{Chentsov1965}, \cite{Chentsov1978},
\cite{MC1991}. Moreover, there exists a natural affine structure on spaces of
  probability measures as discovered by Amari \cite{Amari1980},
\cite{Amari1982} and Chentsov \cite{Chentsov1982}. We refer  the reader  to \cite{MC1991}, \cite{KV1997}  and \cite{AN2000} for
more extensive  historical remarks  and  guide  on other important contributions  in the field.  Such
  structures should be invariant under reparametrizations, and this
  leads us into the realm of differential geometry, the field of
  mathematics that systematically investigates geometric
  invariances. Statistics, however, requires more. There is the
  concept of a sufficient statistic, that is, a mapping between sample
spaces that preserves all information about the parameter
$x$. Therefore, it is natural to also require the invariance of the
geometric structures under sufficient statistics. It is relatively easy
to see that the Fisher metric and the Amari-Chentsov tensor are
invariant, but whether they are the only such invariant structures is
more subtle. This is the question we are addressing and solving in the
present paper. For finite sample spaces, this has been achieved by
Chentsov long ago \cite{Chentsov1978}, see also  Remark \ref{rem:chentsovthm}. The case of infinite sample spaces, however, is
more difficult. The space of probability measures on an infinite
sample space is infinite dimensional, and therefore, standard
constructions from finite dimensional differential geometry may fail. 
 The first successful  approach to apply  techniques of   Banach space   theory  to the space of probability measures on an infinite
sample space has been achieved by Pistone with Sempi \cite{PS1995}
and other coworkers \cite{CP2007}, \cite{GP1998}.  
  However,
  there are technical difficulties, caused for instance by the fact
  that the  topology on the considered Banach manifolds is so strong
  that the space of bounded random variables is not dense in that
  topology \cite[Lemma 2]{CP2007}. \\
Our approach is different. Our essential idea is
that while the space of all probability measures $\Mm(\Om)$ on the sample space
$\Omega$ in general will not carry the required geometric structures,
it can still induce such structures on all finite or infinite dimensional models,
that is, on statistical families with a Banach manifold of
parameters. For that purpose, however, those families need to be
embedded into the space of all measures, and including the embedding
$p$ as part of our notion of a statistical model allows us to treat 
the elements of $M$ as measures on $\Omega$. We can then pull back
tensors from $\Mm(\Om)$ to $M$ and then require the needed regularity
properties not on $\Mm(\Om)$, where we might not be able to define
them, but on $M$, where they can be naturally defined. This leads us to a
notion of a {\em statistical model\/} or {\em statistical manifold\/}
\cite{Lauritzen1987, Le2005, Le2007} as a manifold $M$ equipped with a (Fisher)
metric $g$ and an (Amari-Chentsov) 3-tensor which are induced by an
embedding $p$ into  $\Mm(\Om)$. \\ 
Our approach combines concepts from measure theory, information
theory, and statistics. It thus is situated in information geometry, a new
mathematical field
that recently emerged, where geometric ideas and methods are exploited as principal  tools to  study  mathematical statistics and related problems in  information theory, neural networks, system theory \cite{AN2000}.
Information geometry has also been identified as a natural formalism for complexity theory \cite{Ay,AyJost}. In particular, complex networks 
can be analyzed with tools from information geometry \cite{OlbrichNetworks}.   
 We note that   parameter spaces  in information geometry  are  assumed to be   smooth  manifolds.  This  assumption is  caused  by limitation of  methods of   differential geometry. With recent  extension of  differential geometric methods to  singular spaces, e.g. in \cite{LPV2010}, we hope  to  extend  the field of applications of information  geometry in the future.

The structure of our paper is as follows. In Section \ref{pmm} we introduce the notion of a $k$-integrable parametrized measure model, which  encompasses all known examples in  information geometry  considered by Chentsov, Amari and Pistone-Sempi. We compare our concept with the concept of a geometrically regular  statistical model  proposed by Amari.   
At the end of that Section, we state our Main Theorem \ref{main}. In Section \ref{section-sufficient-statistics} we introduce the notion of sufficient statistics based on the Fisher-Neyman characterization  (Definition \ref{suff}, Lemma \ref{fisherneyman}). We give  a simple proof  that  the Amari-Chentsov structure is invariant under  sufficient statistics (Theorem \ref{th:inv}).
At the end of the section we  discuss  Chentsov's results  on  the
uniqueness of the Fisher metric and the Amari-Chentsov tensor
(Proposition \ref{chentsovcampbell}, Remark \ref{rem:chentsovthm}, Corollary \ref{cor:chentsovcampbell}). At the
end of that Section, we prove our Main Theorem.
 In Section \ref{section-markov} we introduce the notion of a Markov morphism.  A novel aspect of our concept of  
Markov  morphisms between parametrized measure models is  the consideration of smooth maps between  the parameter spaces 
(Definition \ref{equi}, Example \ref{imstat}).  Thus, the geometry of   parametrized measure models is intrinsic. 
We decompose a Markov morphism as a composition of the inverse of a Markov morphism, defined by a sufficient statistic, and a statistic (Theorem \ref{decomp1}).
As a consequence we give a geometric proof of the monotonicity   theory for Markov morphisms  (Corollary \ref{co:monoton}).  Finally, in Section \ref{pistone}, we study  the  relations between $k$-integrable 
parametrized measure models and statistical models  in the Pistone-Sempi theory.

 \section{Parametrized measure models} \label{pmm}

In this section we describe the geometry of spaces of measures
  and of parametrized families of measures. In technical terms, we introduce 
the notion of a $k$-integrable parametrized measure model (Definition \ref{def:gen})
and the notion of tensor fields on them, following the locality and  continuity condition (Definition \ref{df:tensor}, Remark \ref{rem:gen1}).
We  show that 
our  notion of  generalized  statistical models  encompasses all statistical models  considered by Chentsov, Amari,  Pistone-Sempi (Remark \ref{rem:gen1}, Example \ref{ex:statmodel}), and
 we compare  our concept with that by Amari (Remark \ref{compareamari}). 
   
Let  $(\Om, \Sigma)$ be a measurable space. 
Later on, $\Om$ will also have to carry a differentiable
  structure. \\
We consider the Banach space of all signed finite measures on $\Om$
with the total variation ${\| \cdot \|}_{TV}$ as Banach norm. More
precisely, the total variation of such a measure $\mu$ 
is defined as
\[
    {\| \mu \|}_{TV} \; := \; \sup \sum_{i = 1}^n |\mu(A_i)|
\] 
where the supremum is taken over all finite partitions $\Omega = A_1 \dot\cup \dots \dot\cup A_n$ with disjoint sets $A_i \in \Sigma$.
We consider the subset $\Mm(\Om)$  
of all finite non-negative measures on $\Om$, and, with a $\sigma$-finite non-negative
measure $\mu_0$, we also consider the subspace  
\begin{eqnarray*}
   \Ss(\Om, \mu_0) & := & \{\mu = \phi \, \mu_0 \; : \; \phi \in L^1(\Om, \mu_0) \}
\end{eqnarray*}
of signed measures dominated by $\mu_0$.
This space can be identified in terms of the canonical map 
$i_{can}: \Ss(\Om, \mu_0) \to L^1(\Om, \mu_0)$, $\mu \mapsto \frac{d \mu}{ d \mu_0}$. Note that
\[
       {\| \mu \|}_{TV} \; = \; {\left\| \frac{d \mu}{d \mu_0} \right\|}_{L^1(\Omega,\mu_0)},
\]
which implies that $i_{can}$ is a Banach space isomorphism.
Therefore, we refer to the topology of $\Ss(\Om, \mu_0)$ also as the {\em
  $L^1$-topology\/}. This is  independent of the particular
  choice of the reference measure $\mu_0$, because if $\phi \in
  L^1(\Om ,\mu_0)$ and $\psi \in L^1(\Om , \phi \mu_0)$, then $\psi \phi \in
  L^1(\Om,\mu_0)$.  Throughout the paper, we consider the following hierarchy of subsets of $\Ss(\Om, \mu_0)$: 
\begin{eqnarray*}
   \Mm(\Om, \mu_0) & = & \{\mu = \phi \, \mu_0 \; : \; \phi \in L^1(\Om, \mu_0), \;\; \phi \geq 0 \} \\
   \Mm_+(\Om, \mu_0) & = & \{\mu = \phi \, \mu_0 \; : \; \phi \in
   L^1(\Om, \mu_0), \;\; \phi > 0 \} \\
 \Mm^a(\Om, \mu_0) & = & \{\mu = \phi \, \mu_0 \; : \; \phi \in
 L^1(\Om, \mu_0), \;\; \phi \geq 0,\; \mu(\Omega) = {\| \mu \|}_{TV} = a \} \\
   \Pp(\Om, \mu_0) & = & \{\mu \in  \Mm(\Om, \mu_0) \; : \; \mu(\Omega) = {\| \mu \|}_{TV} = 1 \} \\
   \Pp_+(\Om, \mu_0) & = &  \{\mu \in  \Mm_{+}(\Om, \mu_0) \; : \; \mu(\Omega) = {\| \mu \|}_{TV} = 1 \} 
\end{eqnarray*}
In particular, for $\mu =\phi \mu_0 \in \Mm_+(\Om, \mu_0)$,
  i.e., $\phi >0$, $\mu_0$ and $\mu$ have the same null sets and are
  equivalent, that is, $\mu_0 =\phi^{-1} \mu \in \Mm_+(\Om,\mu)$. 
  Thus, we have some kind of multiplicative structure on
  $\Mm_+(\Om, \mu_0)$, and one might hope to generate this via an
  exponential map from the linear structure on $L^1(\Om,\mu_0)$. The
  problem, however, is that if $f\in L^1(\Om,\mu_0)$, then we do not
  necessarily have $e^f \in L^1(\Om,\mu_0)$. When it is, then  $e^f
  \mu_0 \in \Mm_+(\Om, \mu_0)$, but when it is not, the measure
  $e^f\mu_0$ is not well defined. Thus, certain
  infinitesimal deformations are obstructed, that is, cannot be
  integrated into local ones. Of course, this does not happen when
  $\Om$ is finite, the case treated by Chentsov, and this is the
  technical reason why we need to work harder for our main
  result. (Pistone and Sempi have analyzed the underlying topological
  structure, and we shall describe their construction from our
  perspective in Section \ref{pistone}.
  The essential point for an
  intuitive understanding of this topology is that if
  $e^f \in L^1$, then for $0<t<1$, $e^{tf}\in L^p$ for $p=1/t >1$.)
  


In order to avoid this issue and in order to make contact with
the basic construction of parametric statistics, we shall consider
parametrized families of measures, that is,  
maps $M \to \Mm(\Om)$ of  smooth Banach manifolds $M$ into the ``universal measure set''  
$\Mm(\Omega)$ and attempt to pull  geometric structures from $\Mm(\Om)$ back to
$M$  by such maps, which are, in a sense,  similar   to  differentiable maps. Since, however, we may not be able to fully
  define these objects on $\Mm(\Om)$, we shall have to push forward
  tensors from $M$ instead, and integrate them w.r.t. the measures
  $p(x)$ defined by a parametrized family. 
We shall now introduce the technical conditions
needed to realize universal objects on $\Mm(\Om)$ on such parametrized
families.

\begin{definition} \label{df:tensor}
A {\em covariant $n$-tensor  field }on $\Mm(\Om)$  assigns  
to each $\mu \in \Mm(\Om)$  a multilinear map $\tau_\mu: \bigoplus^n
L^n (\Om, \mu) \to \R$ that is continuous w.r.t. the product topology on $\bigoplus^n L^n (\Om, \mu)$.
\end{definition}

In this definition, continuity refers to the continuity of the linear maps
$\tau_\mu$ for fixed
$\mu$. (This is different from 
requiring that $\tau_\mu$ be continuous as a function of $\mu$.) \\
Such objects then will be pulled back to $M$ under a map $p:M\to
\Mm(\Om)$, and they then operate on $n$ vector fields on $M$. When
these vector fields are continuous, their evaluation under the pulled
back covariant tensor field should also be continuous. 
Note that the (Banach) manifold  structure on $M$ defines a 
canonically induced structure of a (Banach) vector bundle $TM \times _{n \, times}  TM \to M$, regarding it as the $n$-fold Whitney sum of the (Banach) vector bundle $TM \to M$. In contrast to direct sums, there is no canonical definition of a topology on the tensor product $T_x^*M \otimes _{n\,  times}  T_x^* M$, whence the tensor product $T^*M \otimes _{n\,  times}  T^* M \to M$ is not a (Banach) vector bundle  if $M$ is infinite dimensional. However, we can define weak continuity of its section as follows.
\begin{definition}\label{def:new} A  continuous  $n$-vector  field  $V_n$  on a Banach manifold $M$  is  a continuous  section  of the
bundle $TM \times _{n \,  times} TM  \to M$.
A section  $\tau$  of the  bundle $T^*M \otimes _{n\,  times}  T^* M$  is  called {\it a  weakly continuous   covariant $n$-tensor},  if  
the value $\tau (V_n)$  is a continuous function for  any  continuous    $n$-vector field
$V_n$ on   $M$.
\end{definition}
For a map $p: M \to \Mm(\Om,\mu)$ the composition $\bar{p} := i_{can}\circ p: M \to L^1(\Omega, \mu)$, 
$x \mapsto \bar{p}(x) := \frac{d p(x)}{d \mu}$, will play a central role. 
Thus, $\bar{p}$ is a map from $M$ to $L^1(\Omega,\mu)$, whence we can consider $\bar{p}$ also as a map $M \times \Omega \to {\mathbb R}, (x, \omega) \mapsto \bar{p}(\omega, x)$ such that
\begin{equation} \label{whatever}
p(x) = \bar{p}(x, \omega) d\mu(\omega).
\end{equation}
Of course, for a fixed $x \in M$, the function $\omega \mapsto \bar{p}(x, \omega)$ is only defined up to changes on a $\mu$-null set in $\Omega$. We refer to a function $\bar{p}: M \times \Omega \to {\mathbb R}$ satisfying (\ref{whatever}) as a {\it density potential}. 
However,
  this notation is slightly misleading, and the infinitesimal tangent
  vector of the family rather corresponds to $\ln \bar{p}(\omega, x)$
  (recall our discussion above of the exponentiation of $f\in
  L^1(\Om,\mu)$, and taking the logarithm of course is the inverse of
  exponentiation.) In particular, the pushforward of a tangent vector
  $V\in T_x M$ is $\p_V\ln \bar p(x, \om)$, and we often simply identify
  $V$ with its pushforward when the map $p$ is fixed in a given context.\\
Our parametrized families of measures will need to satisfy some
  further important technical requirements that we shall now list and
  that will lead us to our technical concept of a 
parametrized measure model.

\begin{enumerate}
\item The parameter space $M$ is a (finite or infinite dimensional) Banach manifold of class at least $C^1$.
\item There is a  continuous mapping $p : M \to \Mm_+(\Om, \mu)$, where the latter is provided with the $L^1$-topology.
\item The composition $\bar{p} = i_{can}\circ p$
is Gateaux-differentiable as a map from the manifold $M$ to the Banach space $L^1(\Om, \mu)$.  
\item The 1-form 
    \begin{equation}\label{1-form}
      A(V)_x:= \int _\Om \p_V   \ln \bar p(x, \om)  \, dp(x), 
    \end{equation}
the  {\it  Fisher quadratic form }
\begin{equation}
   g^F(V, W) _x  : = \int _\Om \p_V   \ln \bar p(x, \om) \p_W
    \ln\bar p(x, \om) \, dp(x)\label{for:fisher}
\end{equation} 
and  the {\it  Amari-Chentsov  3-symmetric tensor}  
\begin{equation} 
   T^{AC} (V,W, X) _x : = \int _\Om \p_V  \ln\bar p(x, \om)
   \p_W  \ln\bar p (x, \om)\p_X  \ln\bar p(x, \om) \ d p(x)\label{for:amari}
\end{equation}
  are well-defined  and  continuous in the sense of Definition \ref{def:new}. 
\end{enumerate}

\begin{remark}\label{rem:hisac}  The name  of Amari and Chentsov   has been attributed to   the tensor $T^{AC}$ in \cite{Le2005} based on the fact that
the  1-parameter family of affine  connections  that are differed
by the   Levi-Civita  connection  of the  Fisher metric  by   the  tensor $T^{AC}$ up to a constant has been discovered  by Chentsov and Amari independently. These connections are also   called {\it the Amari-Chentsov  connections} \cite{Le2005}.  Earlier, in \cite{Lauritzen1987}  Lauritzen  has introduced the notion of   a statistical manifold that is a smooth manifold  equipped with a Riemannian metric  and a 3-symmetric tensor.
\end{remark}



We can now state our general definition of a parametrized measure model. 

\begin{definition} \label{def:gen} (cf. \cite[\S 2 , p. 25]{Amari1987}, \cite[\S 2.1]{AN2000}) Let $k \ge 1$.
A {\it $k$-integrable parametrized measure model} 
is a quadruple $(M, \Om, \mu, p)$ consisting  of a smooth (finite or infinite dimensional) Banach manifold 
$M$ and a continuous map $ p: M \to \Mm_+(\Om, \mu)$ provided with the $L^1$-topology such that there exists a density 
potential $\bar p = {d p \over d\mu}:  M \times \Om \to \R$ satisfying (\ref{whatever}), such that
\begin{enumerate}
\item the function 
$x \mapsto \ln \bar p (x, \om) = \ln\frac{dp(x)}{d\mu}(\om): M \to \R$ is defined and continuously G\^ateaux-differentiable for $\mu$-almost all $\omega \in \Om$, and the correspondence $V \mapsto \partial_V \ln\frac{dp(x)}{d\mu}(\om)$ depends continuously on $V \in TM$ and is linear in each $T_xM$,
\item  for all continuous vector fields $V$ on $M$ the function $
  \omega \mapsto \p _{V} \ln \bar p(x, \om)$ belongs to $L^k(\Om, p(x))$ ;
   moreover,
the function $x \mapsto ||\p _{V} \ln \bar p(x, \om)|| _{L^k(\Om, p(x))}$  is continuous on $M$.
\end{enumerate}
We call $M$ the {\it parameter space} of $(M, \Om, \mu, p)$. 
We call  $(M, \Om, \mu, p)$ {\it  a statistical  model} if 
$p(M)\subset \Pp_+ (\Om, \mu)$.
A $k$-integrable parametrized measure model $(M, \Om, \mu, p)$  is called {\it immersed} if 
$d_x \ln \bar p: T_xM \to L^k(\Om, p(x))$ is injective for all $x \in M$.
\end{definition}

Here the continuous G\^ateaux-differentiability of $\ln \bar p(x, \om)$, for a fixed $\om \in \Om$, is understood  as the  continuity  of the  Gateaux-differential  as a function on $TM$ \cite[chapter I.3]{Hamilton1982}.

\begin{remark}\label{rem:gen1} 
1. Note that, as explained above, the choice  of a reference measure in $\Mm_+(\Om, \mu)$ is immaterial for a 
$k$-integrable parametrized measure model $(M, \Om, \mu, p)$.

2. For a {\it statistical} model, (\ref{1-form}) vanishes
  identically. Recalling the identification of the tangent vector $V$
  on $M$ with its pushforward $\p_V \ln \bar{p}$, 
  this simply means
  \begin{equation}\label{pushvector}
\int_\Om Vd\mu =0. 
    \end{equation}
To obtain (\ref{pushvector})  we argue as follows. For a curve $x(t), \, t \in (-\eps,\eps),$ on $M$ with $\p_t : = \dot x(t) = V((x(t))$  the condition (2) in Definition \ref{def:gen} implies that 
$$ f(t): = \int_{\Om}\p_t  \ln  \bar p  (x(t), \om)\,d p(x(t))$$
is continuous    and hence integrable  over $(-\eps, \eps)$. In particular,  $A(V)_x$ is continuous in $x$. Apply the  Fubini theorem and the condition (1) in Definition \ref{def:gen} we have
$$\int _{-\eps} ^\eps \int_\Om (\p_t \ln \bar p (x(t), \om))p(x(t))d\mu\,dt = \int _{\Om} \int _{-\eps} ^\eps(\p_t \ln \bar p(t, \om)) p (t, \om) dt d\mu  $$
$$=\int _\Om( p(\eps, \om) - p(-\eps, \om))d\mu  = 0.$$

Observe that the above formula  for general $k$-integrable  parametrized  measure models implies
\begin{equation}
\p _V \int _\Omega\, dp(x) = \int _{\Om} \p_V \ln p(x, \om)\, dp(x)\label{tauschen}
\end{equation}
for  all $x \in M$ and  for all tangent vectors $V \in T_x M$.


3. 
For any $k$-integrable parametrized measure model $(M,\Om, \mu, p)$ the composition 
$\bar p = i_{can}\circ p: M \to L ^1 (\Om, \mu)$ is G\^ateaux-differentiable by (\ref{tauschen})    and taking into account
\[
  \int_\Om |\p_V    e ^{ \ln \bar p(x)}| d\mu = \int _{\Om} | \bar p (x) \p_V \ln \bar p (x)| d\mu  = \int _{\Om} |\p_V \ln \bar p(x)| d p(x)< \infty.
\]  

4. Any $3$-integrable parametrized measure model  carries the Fisher quadratic form
and the Amari-Chentsov tensor, which are continuous in the
  sense of Definition \ref{def:new}.  On a $k$-integrable parametrized measure model $(M, \Om, \mu, p)$ the  covariant symmetric $n$-tensor field $T^n(V, \cdots , V) := (\p_V \ln \bar p(x, \om))^n$ 
satisfies the locality and continuity conditions  required in the introduction.

5. In \cite{Chentsov1982} Chentsov considered only   statistical models $(M, \Om_n, \mu_n, p)$ where $M$ is a submanifold in 
$\Pp_{+} (\Om_n, \mu_n)$ and $p $ is the canonical embedding, see also Example \ref{ex:statmodel}.
Amari and all authors  before  Pistone and Sempi considered statistical models $(M, \Om, \mu, p)$   where $M$ is finite dimensional and $p(M) \subset \Pp_{+} (\Om, \mu)$ \cite{AN2000}. Their examples satisfy the conditions in Definition \ref{def:gen}.  

6.  In \cite[Chapter 3]{AJLS2013}  and in \cite[Definition 4.10]{Le2013}  we propose   different  refinements of the notion of a $k$-integrable  parametrized measure model, for
which  the  validity of the  condition (2)    for $k\ge 1$ implies the validity  of the condition (2) for all $1\le p \le k$. Thought the present  notion of a $k$-integrable  parametrized measure model is not as elegant as we wish, it   seems to us closest  to suggestions of   Amari and Cramer, see Remark \ref{compareamari}.
\end{remark} 

\begin{example}\label{ex:statmodel} 1. Let $\Om_n$ be a finite set of $n$  elements and $\mu_n$  a measure of maximal support on $\Om_n$.
It is evident that $\Mm_{+} (\Om_n, \mu_n)$ is diffeomorphic to  $\R^n$. Let $S$ be 
a $C^1$-submanifold  in  $\Pp_{+}(\Om_n, \mu_n)$  and $i_S : S \to \Pp_{+} (\Om_n, \mu_n)$  the canonical embedding. Then 
$(S, \Om_n, \mu_n, i_S)$ is an immersed  $k$-integrable statistical model for all $k \ge 1$. In particular, 
$(\Pp_{+}(\Om_n, \mu_n), \Om_n, \mu_n, Id)$ is a $k$-integrable statistical model. Conversely, for any immersed  $1$-integrable statistical model 
$(M, \Om_n, \mu_n, p)$   the  map $p : M \to \Pp_{+}(\Om_n, \mu_n)$  defines an immersion $M \to  \Pp_{+} (\Om_n, \mu_n)$ between 
differentiable manifolds.

2. If $s: N \to M$ is a  smooth map and $(M, \Om, \mu, p)$ is a $k$-integrable parametrized measure model, then $(N, \Om, \mu, p\circ s)$ 
is a $k$-integrable parametrized measure model.

3. For a measure space $(\Om, \mu_0)$ we define the set
\[
\Mm^{bd}_+(\Om, \mu_0) := \{ \mu = e^f \mu_0 \;: \;f \in L^\infty(\Om, \mu_0)\}.
\]
With the canonical identification $\Mm^{bd}_+(\Om, \mu_0) \ni \mu \mapsto \ln \left( \frac{d\mu}{d\mu_0}\right) \in L^\infty(\Om, \mu_0)$, we may regard $\Mm^{bd}_+(\Om, \mu_0)$ as a Banach manifold, and it is straightforward to verify that the inclusion
\[
p : \Mm^{bd}_+(\Om, \mu_0) \hookrightarrow \Mm_+(\Om, \mu_0)
\]
is $k$-integrable for all $k$.

4. Let $\Om_1, \Om_2$ be smooth manifolds with their Borel $\sigma$-algebras, and let $\kappa: \Om_1 \to \Omega_2$ be differentiable. For a (signed) finite measure $\mu$ on $\Om_1$, we define its push-forward $\kappa_*(\mu)$ as
\[
\kappa_*(\mu) (A) := \mu(\kappa^{-1}(A)), \qquad \mbox{for a Borel subset $A \subset \Omega_2$}.
\]
Moreover, let $\mu_1$ be a Lebesgue measure on $\Omega_1$, i.e., a measure locally equivalent to the Lebesgue measure on ${\mathbb R}^n$, and let $\mu_2 := \kappa_*(\mu_1)$. Then the set $\Omega_2^{sing}$ of the singular values of $\kappa$ is a null set w.r.t. $\mu_2$, and for $\om_2 \in \Omega_2^{reg} := \Omega_2 \backslash \Omega_2^{sing}$, there is a transverse measure $\mu_{\om_2}^\perp$ on $\kappa^{-1}(\om_2) \subset \Omega_1$ such that for each Borel set $A \subset \Om_2$
\[
\int_{\kappa^{-1}(A)} d\mu_1 = \int_A \left( \int_{\kappa^{-1}(\om_2)} d\mu_{\om_2}^\perp \right) d\mu_2(\om_2).
\]
Then the map
\[
\kappa_*: \Mm^{bd}_+(\Om_1, \mu_1) \longrightarrow \Mm_+(\Omega_2, \mu_2)
\]
with $\Mm^{bd}_+(\Om_1, \mu_1)$ from above is a $k$-integrable parametrized measure model for any $k$.
\end{example}

On a 3-integrable parametrized measure model $(M, \Om, \mu, p)$ the pair of the Fisher quadratic form and the Amari-Chentsov tensor will be 
called the {\it  Amari-Chentsov structure}.

\begin{remark}\label{compareamari} We would like to  compare  our concept of a $k$-integrable  parametrized measure model  with the concept of
a geometrical regular  statistical model proposed by Amari,   for instance   in \cite[\S2 ]{Amari1987}. Amari listed  6 properties  a geometrically regular statistical model $\{p (x)\in \Pp_+(\Om, \mu)\}$  must satisfy \cite[A$_1$-A$_6$, p. 25]{Amari1987}.
The condition A$_1$  says that   the domain of parameter $x$ is  homeomorphic to $\R^n$.  The conditions A$_2$  and A$_3$ are   equivalent to   our   condition (2) listed just before  Definition \ref{def:gen}.
The condition A$_4$ requires that $\bar p (x, \om)$ is smooth in $x$ uniformly in $\om$, and moreover the relation (\ref{tauschen}) holds.  The condition A$_5$ requires that  a statistical  model is 3-integrable. 
The last condition A$_6$  requires that the Fisher quadratic form is positive definite.    Amari's conditions are slightly stronger than ours, but in general  our concept agrees with his concept.  Note that similar regularity conditions   have been posed by Cramer \cite[p.500-501]{Cramer1946}, see also \cite[Chapter 2, \S 6]{Kullback}.
\end{remark}


As mentioned above, we consider tensor fields on parametrized measure models $(M, \Omega, \mu,p)$ 
that are inherited from a corresponding field on the ``universal measure set'' 
$\Mm(\Omega)$ in terms of the parametrization $p$. 

Note that we do not impose any strong regularity conditions on tensor fields on $\Mm(\Omega)$. Instead, we assume 
the required regularity and continuity conditions to be satisfied on the pull-back of the field with respect to a parametrization 
$p : M \to \Mm (\Omega)$. 
In addition to these conditions, the existence of a global tensor on $\Mm(\Omega)$ sets some 
compatibility constraints on the associated fields on the class of parametrized measure models $(M, \Omega, \mu,p)$. 
In the following definition we summarize necessary regularity and compatibility conditions for tensor fields, which are, in particular, satisfied in the
case of the Fisher quadratic form and the Amari-Chentsov tensor.
 
\begin{definition}[Locality and continuity condition] \label{def:loc} 
A {\it statistical} covariant continuous $n$-tensor field  $A$      
assigns  to  each  parametrized measure model $(M, \Om, \mu, p)$
a   {\it   continuous (in the sense of Definition \ref{def:new})}  covariant $n$-tensor field $A|_{(M, \Om, \mu, p)}$ on $M$   (cf. Definition  \ref{df:tensor}).   
A statistical covariant  continuous  $n$-tensor field  $A$  is called {\it  local}  if 
 there is  a pointwise continuous  covariant $n$-tensor field
$\tilde A$ on $\Mm(\Om)$ with the following property
\begin{equation}
A|_{(M, \Om, \mu_0, p)} (V_1(x), \cdots, V_n(x)) = \tilde A_{p(x)}(\p _{V_1} \ln \bar p(x),  \cdots,\p_{V_n}\ln \bar p(x)) .\label{for:local}
 \end{equation}
\end{definition}
 In particular, this means that the value depends only on $p(x)$,
  but not on the manifold $M$ defining the parametrized family of
  which $p(x)$ is a member.
\begin{remark}\label{re:comp1} 
1. Assume that   $A$  is a local  statistical  covariant  continuous
 $n$-tensor field. Using Example \ref{ex:statmodel}.3 
  we note that  
  there exists at most one  
point-wise continuous $n$-tensor field $\tilde  A$ on  $\Mm(\Om)$  such that $A$ is defined by $\tilde A$  as in  (\ref{for:local}).  
Thus, in order to  define $A$  it suffices to determine the associated
point-wise continuous $n$-tensor field $\tilde  A$ on  $\Mm(\Om)$ and then
verify if  the  original statistical field $A$ is  continuous.

Condition (\ref{for:local})  holds  for the Fisher quadratic  form field and the Amari-Chentsov tensor  field. The choice of $\p_V \ln \bar p (x)$ is also related to the G\^ateaux-differentiability of $p$ (Remark \ref{rem:gen1}.3). We choose $L^n (\Om, p(x)) $   as a natural condition for the value $\p_{V_n}\ln \bar p(x)$ since  it is  a natural   extension  of the condition for the existence of the  Fisher quadratic form and the Amari-Chentsov tensor on a parametrized measure model.
 
2. The  locality  and continuity condition holds obviously for  tensor fields on statistical models associated  with finite sample spaces as  in the Chentsov work \cite{Chentsov1982}.
 
3.  In \cite{Le2005} and \cite{Le2007}, L\^e proved the following   variant of the locality condition, which has been asked by Lauritzen \cite{Lauritzen1987} and Amari-Nagaoka \cite{AN2000}.  For  any  statistical manifold $(M, g, T)$ there exist
a finite sample space $\Om_n$  provided with a dominant measure $\mu_n$   and
an  immersion $p: M \to  \Mm(\Om_n, \mu_n)=  \Mm (\Om_n)$ such that the statistical structure 
$(g, T)$ is induced from the Amari-Chentsov structure  
on $(\Mm (\Om_n, \mu_n), \Om_n, \mu, Id)$ via $p$.  
\end{remark}

Our main theorem uses the notion of a sufficient statistic and the associated invariance property. As already stated in the introduction, sufficient statistics are important  transformations between parametrized measure models, since  they preserve the information of the underlying models. Although we introduce the corresponding definitions later in the paper, we present our main theorem already here so that its main structure guides the arguments and motivates further results of the paper.
 
 \begin{theorem}[Main Theorem] \label{main} (1) Assume that $A$ is a  local  statistical   continuous   1-form field. 
 If $A$ is invariant under  sufficient statistics then  there is a continuous function $c: \R \to \R$  such that  for all  finite measures 
 $\mu$ on $\Om$ and for all  $V \in L^1 (\Om, \mu)$
 we have 
\[
    \tilde{A}_\mu  ( V) = c(\int_\Om d\mu)\cdot \int _\Om V d\mu.
\]    
 In particular, recalling (\ref{pushvector}), there is  no weakly continuous 1-form field  on  statistical models that is invariant under  sufficient statistics. On  a parametrized measure model $(M, \Om, \mu, p)$   the field $A$ is expressed as follows
 \begin{equation}
 A (V)_x  = c(\int _\Om dp(x)) \cdot  \p_V ( \int _\Om  dp (x)). \label{for:main1}
 \end{equation}
 
 (2)  Assume that $F$ is a  local  statistical  continuous quadratic form field. 
 If $F$ is invariant under  sufficient statistics then  there  are  continuous  functions $f, d: \R \to \R$  such that  $F(x)   = f(\int_\Om d p(x))g^F(x) +  d(\int_\Om d p(x))A(x) ^2$, where $A$ is the  field in (1)  with $c=1$  and $g^F$ is the Fisher quadratic  form. In particular,  the Fisher quadratic form  is the unique up to a constant  weakly continuous quadratic  form field  on statistical models
 that is invariant under sufficient statistics.
 
 (3)  Assume that $T$ is a  local   statistical  continuous covariant symmetric 3-tensor field. 
 If $T$ is invariant under  sufficient statistics then there is a  continuous  function $t: \R \to \R$  such that  
 $T(x)   = t(\int_\Om p(x))T^{AC}(x) + A_1(x)^3 + A_2(x) \cdot g^F(x)$, where  $A_1, A_2$  are the fields described in (1), $g^F$  and $T^{AC}$ are the Fisher  quadratic form and the Amari-Chentsov tensor  respectively. In particular, the Amari-Chentsov tensor is the unique up to a constant  weakly  continuous 3-symmetric tensors field on statistical models  that is invariant under sufficient statistics.
\end{theorem}
 
 Campbell noticed that  the Fisher metric  on  parametrized measure models  associated with a finite sample space $\Om_n$  coincides with the Shahshahani metric \cite{Campbell1986}, which is important in mathematical biology and game theory \cite{Shahshahani1979}. It is interesting to find   applications in this direction of the Fisher metric and   other natural  metrics  on generalized  statistical models  described in the Main Theorem.
  
 \section{Sufficient statistics  and the  Amari-Chentsov structure} \label{section-sufficient-statistics}
{\it A statistic $\kappa$} is a measurable map between a measure
  space $(\Om_1, \mu_1)$ and a measurable space $\Om_2$.  One of the
  most important properties of the  Fisher  quadratic form and the
  Amari-Chentsov tensor is the invariance of  these structures under
  statistics $\kappa: \Om_1 \to \Om_2$  that are sufficient (a notion
  introduced by Fisher in 1922) for the parameter
$x\in (M, \Om_1, \mu, p)$ (Definition \ref{suff}, Theorem \ref{th:inv}). In other words, the Fisher quadratic form  and the Amari-Chentsov tensor on
$(M,\Om_1, \mu, p)$  and $(M, \Om_2, \kappa_*(\mu), \kappa_*(p))$  coincide, if $\kappa$ is sufficient. Sufficient statistics  
represent important  transformations between parametrized measure models, since  they preserve the information
of the underlying models.  Thus one  wishes to know whether there are other  quadratic forms and 3-symmetric tensors  on parametrized measure models  which  are invariant under sufficient statistics. This question has been
solved by Chentsov  in the negative for statistical models
associated with finite sample spaces \cite{Chentsov1982}, see Proposition
\ref{chentsovcampbell}. However, one naturally wishes to
  consider {\it infinite} sample spaces $\Omega$, and in this case the space
  of measures becomes {\it infinite dimensional}, and the topological
  aspects then become more subtle. More precisely, the   main difficulty  for an extension of   the Chentsov theorem to  all parametrized measure models  
is caused  by  two facts. Firstly, a statistical model associated  with a finite sample space can be  regarded locally as  
a submanifold in a universal  statistical  model $(\Pp_+ (\Om_n, \mu_n), \Om_n, \mu_n)$, 
which is a finite-dimensional open simplex (Example \ref{ex:statmodel}).  In this case, it suffices to consider  the  Fisher metric,  the Amari-Chentsov  tensor   and other tensor fields on this  open simplex.  Secondly, the structure of  sufficient statistics  associated with   the considered  statistical models   can be described in terms of  Markov congruent embeddings \cite{Chentsov1982}, see also  our discussion at the end of Section 4.   It is not easy to  generalize these facts to statistical models associated with infinite sample space, since, in particular, there is no canonical smooth structure on the set $\Mm_{+}(\Om, \mu)$ of all measures  equivalent to $\mu$, or on the set 
$\Mm(\Om, \mu)$ of all measures dominated by $\mu$.

In this section, we first give  a simple proof  that  the Amari-Chentsov structure is invariant under  sufficient statistics (Theorem \ref{th:inv}).
We  also give a geometric  proof of the  Fisher-Neyman factorization theorem which characterizes  a sufficient statistic 
$\kappa : (\Om_1, \mu_1) \to \Om_2$ under the assumption that
$\kappa: \Om_1 \to \Om_2$ is a smooth map (Theorem \ref{ssta}).
Using Theorem \ref{ssta} we present a proof of the monotonicity
theorem  (Theorem
\ref{th:monoton}).  We also consider  examples of sufficient statistics, which are associated with  Markov congruent embeddings from 
$\Mm_{+}(\Om_n, \mu_n)$ to $\Mm_{+}(\Om_m, \mu_m)$ (Example
\ref{ex:suff1}). Using them we  discuss  Chentsov's results on
geometric structures  which are invariant under sufficient statistics
between  finite sample spaces (Proposition \ref{chentsovcampbell},
Lemma \ref{lem:scale}). We shall then be in a position to prove
  our Main Theorem. 
\medskip

For a  measurable map $\kappa:(\Om_1,\mu_1) \to \Om_2$ let us denote by $\kappa _* (\mu_1)$  the push-forward measure  on $\Om_2$.

\begin{definition}\label{suff}  (cf. \cite[(2.17)]{AN2000},
  \cite[Theorem 1, p. 117]{Borovkov1998}) Assume that $(M, \Om_1, \mu_1, p_1)$ is a $k$-integrable parametrized measure model and $\Om_2$ is a  measurable space. A statistic $\kappa: (\Om_1, \mu_1) \to \Om_2$ is said to be {\it sufficient for the parameter $x\in M$} if   there exist a function  $s: M \times \Om_2 \to \R$   and a function $t  \in L^1(\Om_1, \mu_1)$ such that  for all $x \in M$ we have 
$s(x, \om_2) \in L^1 (\Om_2, \kappa_*(\mu_1))$  and  
\begin{equation}
\bar p_1(x, \om_1) =  s(x, \kappa(\om_1))t(\om_1) \hspace{1cm} \mu_1-a.e.\, .  \label{suff1}
\end{equation} 
\end{definition}


\medskip

\begin{remark}\label{his1}  Definition \ref{suff}  is a version of   the Fisher-Neyman  characterization  theorem, which states that
a statistic  is sufficient for the parameter $x \in M$ if and only if (\ref{suff1}) holds. 

\end{remark}

A  measurable map $\kappa:(\Om_1,\mu_1) \to \Om_2$ transforms  a parametrized measure model  $(M, \Om_1, \mu_1, p_1)$ into   the parametrized measure model $(M_1, \Om_2, \kappa_*(\mu_1), \kappa_*(p_1))$ whose  density potential $\kappa_*(\bar p_1)$   is defined by 
\begin{equation}
\kappa_*(\bar p_1) : = \frac{d \kappa_* (p_1 )}{d\kappa_* (\mu_1) }.\label{defe1a}
\end{equation}


\begin{lemma}\label{fisherneyman}  A statistic $\kappa : (\Om_1, \mu_1) \to \Om_2$  is  sufficient for the parameter $x\in M$ if and only if the function 
$$ r (x, \om_1) : = \frac{\bar p_1 (x, \om_1)}{ \kappa_* (\bar p_1)(x, \kappa(\om_1))}$$
does not depend on $x$ for almost all $\om_1 \in (\Om_1, \mu_1)$.
\end{lemma}

\begin{proof}  The  ``if"  part of Lemma \ref{fisherneyman} is obvious.  Now  we assume that (\ref{suff1}) holds, i.e. 
$\bar p_1(x,\om_1) = s(x, \kappa(\om_1))\cdot t(\om_1)$  for all $x\in M$ and almost everywhere  on $(\Om_1, \mu_1)$. 
 Then for all $x \in M$ and almost all $\om_1 \in (\Om_1, \mu_1)$  we have
\begin{equation}
\kappa_* (\bar p_1)(x,\kappa(\om_1)) = \kappa _* (t)(\kappa(\om_1)) \cdot  s(x, \kappa(\om_1)) .\label{ssta4}
\end{equation}
From   (\ref{ssta4}) we obtain  for  all $x\in M$  
\begin{equation}
r(x, \om_1)  = \frac{ t( \om_1) s(x, \kappa(\om_1))}{\kappa _* (t)(\kappa(\om_1)) \cdot  s(x, \kappa(\om_1)) } =\frac{t(\om_1)}{\kappa _* (t)\kappa(\om_1)} \hspace{1cm}  \mu_1-a.e.\: .
\end{equation}
This completes the proof  of Lemma \ref{fisherneyman}.
\end{proof}

We get immediately

\begin{corollary}\label{cor:reg} Assume that $\kappa: \Om_1 \to \Om_2$ is a sufficient
statistic  for the parameter  $x\in M$ where $(M, \Om_1, \mu_1, p_1)$ is a $k$-integrable
parametrized measure model. Then $(M, \Om_2, \kappa_*(\mu_1), \kappa_*(p_1))$ is also a $k$-integrable parametrized measure model.
\end{corollary}

Let $\kappa : (\Om_1, \mu_1) \to (\Om_2, \mu_2)$ be a statistic  and $(M, \Om_1, \mu_1, p_1)$  a $k$-integrable parametrized measure 
model.  
 The  Fisher quadratic  form $\tilde g ^F$ on 
the transformed parametrized measure model
$(M, \Om_2, \kappa_*(\mu_1), \kappa_*(p_1))$ is defined by
\begin{equation}
\tilde g ^F (V, V)_x = \int_{\kappa(\Om_1)}(\p_V \ln (\kappa_*(\bar p_1)(x,\om_2))) ^2d\kappa_*(p_1(x))  .\label{defis}
\end{equation}

\begin{theorem}\label{th:inv}  If  a statistic $\kappa$ is sufficient for the parameter $x\in M$, then the   
Amari-Chentsov structure  transformed by $\kappa$  is  equal  to the original structure.
\end{theorem}

\begin{proof} Assume that a statistic  $\kappa$ is  sufficient  for the parameter $x \in M$.
By Lemma \ref{fisherneyman} we have  for all $x\in M$
\begin{equation}
 p_1 (x, \om_1)   = r(\om_1) \kappa_* (p_1(x)) (\kappa (\om_1))\hspace{1cm} \mu_1-a.e.\: .\label{fisherneyman2}
\end{equation}
Hence  for all $x\in M$ and all $V \in T_x M$
\begin{equation}
\p _V\ln p_1(x, \om_1) =  \p_V \ln \kappa_* (p_1(x))(\kappa (\om_1))  \hspace{1cm} \mu_1-a.e.\: .
\end{equation}
It follows  for all $x\in M$ and all $V \in T_x M$
\begin{eqnarray*}
   g^F (V, V) _x & = & \int _{\Om _1} (\p_V \ln \kappa_* (p_1(x))(\kappa (\om_2)))^2 r(\om_1) \kappa_* (p_1(x)) (\kappa (\om_1))d\mu_1 \\
                         & = & \tilde g ^F (V, V) _x .
\end{eqnarray*}
This proves the invariance of the Fisher   metric under  sufficient statistics.  The invariance of  the Amari-Chentsov tensor under  sufficient statistics  is proved in the same way.
\end{proof}

\begin{corollary}\label{diffin}  Assume that $\Om$ is a differentiable  manifold provided with the Borel $\sigma$-algebra. The Amari-Chentsov  structure on any $k$-integrable parametrized measure model $(M, \Om, \mu, p)$ is invariant under the  action of the diffeomorphism group of $\Om$.
\end{corollary}

\begin{remark}\label{sufhis} The first known  variant  of Theorem \ref{th:inv} is the second  part of the   monotonicity   Theorem (Theorem \ref{th:monoton}). 
The invariance  of the Amari-Chentsov structure on statistical models associated with finite sample spaces under sufficient statistics  has been discovered first by Chentsov \cite{Chentsov1982}.
\end{remark}

In what follows we  interpret the function $r(x, \om_1)$  assuming that  $\Om_1$  and $\Om_2$ are smooth manifolds supplied  with the Borel 
$\sigma$-algebra  and $\kappa$ is smooth. Furthermore, we assume that $\mu_1$ is  dominated by  a Lebesgue measure on $\Om_1$, i.e.  a measure that is locally equivalent to the Lebesgue measure  on $\R^n$. Then the set $\Om _2 ^{sing}$ of  singular values of $\kappa$ is a null set in $(\Om_2, \kappa_*(\mu_1))$. Let $\om_2$ be a regular value of $\kappa$. Then 
$\kappa^{-1}(\om_2)$ is  a smooth submanifold of $\Om_1$.  
Furthermore, any   sufficiently small   open neighborhood $U_\eps (\om_2)\subset \Om_2$ of $\om_2$ consists only of regular values of $\kappa$. Without loss of generality we assume that    the   preimage $\kappa^{-1} ( U_\eps (\om_2))$ is a direct product 
$U_\eps (\om_2) \times \kappa^{-1} (\om_2)$, which is the case if  $U_\eps(\om_2)$ is diffeomorphic to a ball. 
The measure  $\mu_1$ (respectively, $p_1(x)$) on the source space  and the induced  measure
$\kappa_*(\mu_1)$ (respectively, $\kappa_*(p_1(x))$) on the target space define a  ``vertical"   measure  $\mu^\perp _{\om_2}$, which depends on $\mu_1$,  on each fiber $\kappa^{-1} (\om_2)$  by the following formula: 
\begin{equation}
    d\mu^\perp_{\om_2} (\mu_1, y): = {d\mu_1 (\om_2, y)\over  d\kappa_* (\mu_1) (\om_2)} \label{def3}
\end{equation}
for all $y \in  \kappa^{-1} (\om_2)$. (Respectively, we replace $\mu_1$ by $p_1(x)$ in
the LHS and RHS of (\ref{def3})).  Here we identify a point $(\om_2, y\in \kappa^{-1} (\om_2))$ with the image of $y$ in $\Om_1$  via the inclusion 
$f ^{-1} (\om_2) \to \Om_1$. Note that $d\mu_{\om_2} ^\perp (\mu_1, y)$ is well-defined only if $\om_2 \in \kappa(\Om_1)$.

\begin{lemma} \label{lem:regcond} Assume that  the value $\om_2$ of a statistic $\kappa$ is regular. Then $\mu^\perp_{\om_2}(\mu_1)$  is a probability  measure on $\kappa^{-1} (\om_2)$ for any finite measure $\mu_1$ on $\Om_1$. 
\end{lemma}

\begin{proof} 
We need to show that
\begin{equation}
    \int_{\kappa^{-1} (\om_2)}d\mu ^\perp_{\om_2}(\mu_1, y) = 1. \label{def3a}
\end{equation}

Let $g$ be a Riemannian metric on $\Om_2$. Denote by $D_\eps(\om_2)$  the disk with center  at $\om_2$ and of radius $\eps$. 
Using (\ref{def3}) and Fubini's formula  we obtain
\begin{equation}
    \int_{D_\eps(\om_2)} d\kappa _*(\mu_1)\int_{\kappa^{-1} (\om_2)}d\mu ^\perp_{\om_2} (\mu_1, y) = \int_{\kappa^{-1}(D_\eps(\om_2))} d\mu_1.\label{def3b}
\end{equation}
Taking into account 
\begin{equation}
    \int_{\kappa^{-1}(D_\eps(\om_2))} d\mu_1 = \int_{D_\eps(\om_2)}d\kappa_* (\mu_1),\label{def3c}
\end{equation}
we derive  from (\ref{def3b})
\begin{equation}
   \int _{\kappa^{-1} (\om_2)} d\mu ^\perp _{\om_2}(\mu_1, y)  = 
   \lim _{\eps \to 0} \frac{\int _{D_\eps (\om_2)} d\kappa_* (\mu_1)}{\int _{\kappa^{-1} (D_\eps (\om_2))} d\mu_1  } = 1.\label{for:cond1}
\end{equation}
This proves (\ref{def3a}) and Lemma \ref{lem:regcond}. 
\end{proof}

\begin{remark}\label{condp}
The measure    $\mu^\perp_{\om_2}$  is the  conditional distribution $(d\om_1| \om_2)$ of the variable (elementary event) $\om_1$ subject to the condition $\kappa = \om_2$.  In general,  a conditional distribution $ (d\om_1|\om_2)$ of the variable  $\om_1$ subject to  condition $\kappa = \om_2$  can be defined   for  measurable mappings, which need not be smooth. We refer to \cite [p. 81]{Kallenberg2001},  \cite[p. 106]{Borovkov1998} for a definition of  a conditional distribution in a general case.
\end{remark}


\begin{theorem}\label{ssta}  Assume that $\Om_1$  and $\Om_2$  are smooth manifolds  supplied with Borel $\sigma$-algebras and $\mu_1$ is a measure on $\Om_1$ dominated by a Lebesgue  measure. Let  
$(M, \Om_1, \mu_1, p_1)$  be  a $k$-integrable parametrized measure model. A   smooth   statistic $\kappa: (\Om_1, \mu_1) \to \Om_2$ is sufficient for the parameter  $x\in M$ if and only if  the conditional
distribution  $\mu _{\om_2}^\perp (p_1(x))$  defined on the set of regular values   $\om_2$ of $\kappa$
 is independent of $x\in M$.
\end{theorem}
\begin{proof} 
 Representing a point $\om_1$ by the pair $(\kappa(\om_1), y)$, $y \in \kappa^{-1} (\kappa(\om_1))$, we write
\begin{equation}
d\mu ^{\perp} _{\kappa(\om_1)} ( p_1(x), y)  = \tilde \mu ^{\perp}_{\kappa(\om_1)}(x,y) d\mu ^{\perp}_{\kappa(\om_1)} (\mu_1,y).\label{ssta2}
\end{equation}
Observe that  (\ref{ssta2}) is equivalent to the following
 \begin{equation}
\bar p_1(x, (\kappa(\om_1), y)) = \tilde\mu^\perp _{\kappa(\om_1)} (x, y) \kappa_*(\bar p_1)(x, \kappa(\om_1)).\label{for:rdensity1}
\end{equation}
(\ref{for:rdensity1}) implies that $\tilde\mu^\perp _{\kappa(\om_1)} (x, y)$  coincides with $ r(x, \om_1)$. Now we obtain Theorem \ref{ssta} from Lemma \ref{fisherneyman} immediately.
 \end{proof}

Using  Lemma \ref{fisherneyman} and Theorem \ref{ssta} we will  present a proof of the monotonicity theorem (Theorem \ref{th:monoton}), 
  which characterizes  sufficient statistics  in terms of the Fisher information metric. 

\begin{theorem}\label{th:monoton}   (Monotonicity theorem, cf. \cite[Theorem 2.1]{AN2000}). Assume that $\Om_1$  and $\Om_2$  are  smooth manifolds provided with Borel  $\sigma$-algebra and $\mu_1$ is a Lebesgue measure on $\Om_1$.  Let $(M, \Om_1, \mu_1, p_1)$ be a $k$-integrable parametrized measure model  and $\kappa: \Om_1 \to \Om_2$  
a statistic. Denote by $\tilde g^F$  the Fisher metric on the  transformed  parametrized measure model  $(M, \Om_2, \kappa_*(\mu_1), \kappa_*(p_1))$. For each $x\in M$ and each $V \in T_xM$ we have 
\begin{equation}
\tilde g ^F (V, V)_x \le g ^F (V, V) _x.\label{eq:mon}
\end{equation}
Inequality (\ref{eq:mon}) becomes an equality for all $x \in M$  and for all $V\in T_x M$ if  and only if the statistic $\kappa$ is sufficient  
for the parameter $x\in M$.
\end{theorem}

\begin{proof} 
Denote by $\Om ^{reg}_2$ the set of regular values of $\kappa$. Using (\ref{def3}),  we obtain
\begin{equation}
 g ^F (V, V)_x =\int_{\Om_2 ^{reg}}d\kappa_*(p_1(x))\int_{\kappa^{-1}(\om_2)} (\p _V \ln \bar p_1(x,y))^2\, d\mu_{\kappa(\om_1)} ^\perp(p_1(x), y).\label{for:monoton1}
\end{equation}
Recall that
\begin{equation}
\tilde g^F(V, V)_x =\int_{\Om_2^{reg}}(\p _V \ln \kappa_*(\bar p_1)(x, \om_2))^2 d\kappa_*(p_1(x)). \label{for:monoton2}
\end{equation}
To prove Theorem \ref{th:monoton}, comparing (\ref{for:monoton1}) with (\ref{for:monoton2}), it suffices to show   that for each $x\in M$ and for each
  $\om_2 \in \Om^{reg}_2$ the following inequality holds
\begin{equation}
\int_{\kappa^{-1}( \om_2)}(\p _V \ln \bar p_1(x,y))^2\, \mu_{\om_2} ^\perp(p_1(x), y)\ge (\p_V \ln \kappa_*(\bar p_1)(x,\om_2))^2, \label{for:monoton3}
\end{equation}
and the equality holds for all $x\in M$ and all  regular  values $\om_2$ if and only if $\kappa$ is sufficient for the parameter
$x\in M$.

Taking into account (\ref{for:rdensity1}) and Lemma \ref{lem:regcond}, we note that (\ref{for:monoton3}) is equivalent to the following inequality
\begin{eqnarray}
\int_{\kappa^{-1}( \om_2) }(\p _V \ln \kappa_*(\bar p_1)(x, \om_2) + \p_V \ln  \tilde \mu _{\om_2}^\perp(x, y))^2\, d\mu_{\om_2} ^\perp(p_1(x), y)\ge  \nonumber\\
\int_{\kappa^{-1}( \om_2)}(\p_V \ln \kappa_*(\bar p_1)(x, \om_2)) ^2\,d\mu_{\om_2} ^\perp(p_1(x), y). \label{for:monoton4}
\end{eqnarray}

\begin{lemma}\label{lem:mup1} For all $x\in M$ we have
\begin{equation}
\int_{\{ y\in \kappa( \om_1)\} }\p_V \ln \tilde \mu_{\kappa(\om_1)}^\perp (x, y) d\mu_{\kappa(\om_1)}^\perp (p_1(x), y)= 0.\label{for:mup1}
\end{equation}
\end{lemma}
\begin{proof} Writing $\mu_{\kappa(\om_1)}^\perp (p_1(x))= \tilde \mu ^\perp _{\kappa(\om_1)}(x, y)\mu _{\kappa(\om_1)} ^\perp (\mu_1)$, we observe that (\ref{for:mup1}) is a consequence of  the following identity for
all $x\in M$:
$$\int_{\{ y\in \kappa( \om_1)\} } \tilde \mu_{\kappa(\om_1)}  ^\perp(x, y) d\mu _{\kappa(\om_1)} ^\perp (\mu_1, y) = 1,$$
whose validity follows from Lemma \ref{lem:regcond}.
\end{proof}

Clearly (\ref{for:monoton4}) follows from Lemma \ref{lem:mup1}, since  $\p _V \ln \kappa_*(\bar p_1)(x, \om_2)$ does not depend on $y$.
Note that (\ref{for:monoton4}), and hence (\ref{for:monoton3}), becomes an equality if and only if $\mu ^\perp _{\kappa(\om_1) }(p_1(x))$ is independent  of $x$. By Theorem \ref{ssta} the last condition is equivalent to the
sufficiency of the statistic $\kappa$  for the parameter $x\in M$. This proves   Theorem \ref{th:monoton}. 
\end{proof}

\begin{remark}\label{cramer}   Assume that  a statistic $\kappa$ is smooth. 
Denote by $\hat g ^F_{\om_2}$ the Fisher quadratic form
on the statistical model $\mu^\perp_{\om_2}(p_1(x), y)$ with respect to the reference measure $\mu ^\perp_{\kappa(\om_1)}(\mu_1, y)$  
as in (\ref{ssta2}). Taking into account (\ref{for:monoton1}), (\ref{for:monoton2}) and (\ref{for:mup1}) we obtain immediately the following 
equality   for all $x\in M$  and all $V \in T_x M$ (cf. \cite[Theorem 2.1]{AN2000})
\begin{equation}
g ^F(V, V)  = \tilde g ^F(V, V) +\int_{\Om_2} \hat g^F _{\om_2}(V, V)d\kappa_*(p_1(x)).\label{for:cramer}
\end{equation}
The integral in the RHS of (\ref{for:cramer}) is called the information loss \cite[p.30]{AN2000}.
\end{remark}

\begin{example}\label{ex:suff1}  Let $\Om_n$  be a finite  set of $n$  elements $E_1, \cdots, E_n$.  Let $\mu_n$  be the probability
distribution on $\Om_n$  such that $\mu_n (E_i) = 1/ n$  for $i \in [1,n]$.  Clearly,
the space $\Pp(\Om_n, \mu_n)$ consists of all probability distributions $p$ on $\Om_n$
which can be represented as
\begin{equation}
    p (E_i) = f(E_i) \mu_n \text{ for } i \in [1,n]\label{capmun}
\end{equation}
for some non-negative  function $f: \Om_n \to \R$ such that $\sum_{i = 1} ^n f(E_i) = n$.
Denote
by $E_i^*$ the Dirac measure on $\Om_n$  concentrated at $E_i$.  The space
$\Mm_{+}(\Om_n, \mu_n)$ of  measures equivalent to $\mu_n$ consists of all measures $p = \sum_{i=1}^n p_i E_i^*,
 p_i > 0$, so  it is the positive cone $\R ^n_+$.
Let $n \le m< \infty$.
Let $\{ \hat F_1,\cdots , \hat F_n \}$ be a partition  of the set $\Om_m :=\{F_1, \cdots, F_m\}$ into disjoint subsets. Denote this partition by $\bar \kappa$. We associate   
$\bar \kappa$ with a map $\kappa: \Om_m  \to \Om_n$   by
setting
$$\kappa (x): = E_i \text{ if  }  x \in \hat F_i .$$
We identify  $\Mm_{+}(\Om_m, \mu_m)$  with $\R ^m _+$ which  is  the convex hull of  the Dirac measures $F^*_j, j \in [1,m]$.  Recall that  a linear mapping   
$\Pi : \R^n  \to \R^m, \, \Pi (E^* _k) : = \sum _{j=1}^m\Pi_{kj}F^*_j,$  is called {\it a  Markov  mapping}, if $\Pi_{ij} \ge 0$ and  $\sum_{j =1} ^m \Pi_{kj} = 1$ 
(cf.  Example \ref{ex1}). Following Chentsov \cite[p. 56 and Lemma 9.5, p. 136]{Chentsov1982}, we  call $\Pi$  {\it   a  Markov congruent  embedding
subjected to a  partition $\bar \kappa$} if
\begin{itemize}
\item $F_j \not\in \kappa ^{-1} (E_i) \;\; \LRA \;\; \Pi(E _i ^*) (F_j)  = 0$, 
\item $\Pi (E_i ^*) \not = 0$ for all $i \in [1,n]$.
\end{itemize}
Note that $\Pi(\Mm(\Om_n, \mu_n)) \subset \Mm(\Om_m, \mu_m)$. The restriction of $\Pi$ to $\R^n_+ = \Pp_+(\Om_n, \mu_n)$   as well to  $\Mm_+(\Om_n, \mu_n)$
is also denoted  by $\Pi$. 

\begin{proposition}\label{congruent1} 1. Let $\Pi: \Pp_+(\Om_n, \mu_n) \to \Mm(\Om_m, \mu_m)$ be  the restriction of a Markov mapping such that 
$(\Pp_+(\Om_n, \mu_n), \Om_m, \mu_m, \Pi)$  is an immersed statistical model of dimension $n-1$. 
A statistic $\kappa: \Om_m \to \Om_n$ is sufficient for the parameter
$x\in (\Pp_+(\Om_n, \mu_n), \Om_n, \mu_n, \Pi)$, if $\Pi$ is
a Markov congruent embedding subjected to $\kappa$.

2.  Let $\Pi: \Mm_+(\Om_n, \mu_n) \to \Mm(\Om_m, \mu_n)$ be  the restriction of a Markov mapping such that 
$(\Mm_+(\Om_n, \mu_n), \Om_n, \mu_n, \Pi)$  is an immersed parametrized model of dimension $n$. 
A statistic $\kappa: \Om_m \to \Om_n$ is sufficient for the parameter
$x\in (\Mm_+(\Om_n, \mu_n), \Om_n, \mu_n, \Pi)$, if $\Pi$ is
a Markov congruent embedding subjected to $\kappa$.
\end{proposition}

\begin{proof} The first assertion  of  Proposition \ref{congruent1}   follows directly from  the  Chentsov   results \cite[Lemma 6.1, p.77 and Lemma 9.5, p.136]{Chentsov1982}.

The second  assertion  of Proposition \ref{congruent1} is a consequence of the  first assertion and the following

\begin{lemma}\label{lem:scale} Assume $(M, \Om, \mu, p)$ is a   parametrized measure model and $\kappa: \Om \to \Om ' $ is sufficient  for the parameter $x \in (M, \Om, \mu, p)$. Then  $\kappa$ is also sufficient   for the parameter $x \in (M\times (0,1), \Om, \mu, \hat p(x, t): = t p (x))$.
\end{lemma} 
\begin{proof}[Proof of Lemma \ref{lem:scale}]  Since $\kappa_* (t\mu) = t\kappa_*(\mu)$ for any finite measure $\mu$ on $\Om$ and $ t \in \R ^+$,  we get
$$\frac{d(t p(x))}{d\kappa_*(tp(x))} = \frac{dp(x)}{d\kappa_*(p(x))}.$$
Taking into account Lemma \ref{fisherneyman}, this  proves Lemma \ref{lem:scale}.
\end{proof}
This  completes  the proof  of  Proposition \ref{congruent1}.
\end{proof}
\end{example}

Since  $\kappa_* \circ \Pi  = Id$ for Markov congruent embeddings $\Pi$, using Theorem \ref{th:inv} we obtain immediately

\begin{corollary}\label{co:chentsov}  Let  $\Pi: \Mm(\Om_n, \mu_n) \to \Mm(\Om_m, \mu_m)$ be a Markov congruent embedding. 
Then  the Amari-Chentsov structure on $\Mm_{+}(\Om_n, \mu_n)$ coincides with the Amari-Chentsov structure on 
$(\Mm_{+}(\Om_n, \mu_n), \Om_m, \mu_m, \Pi)$.
\end{corollary}

\begin{remark}\label{rem:chentsovthm} A variant of Proposition \ref{congruent1}  has been proved by Chentsov \cite[Lemma 6.1, p.77 and Lemma 9.5, p.136]{Chentsov1982}, see also Proposition \ref{pro:immark} below.  It  plays a decisive role in the Chentsov theorem \cite{Chentsov1982} on geometric structures on statistical models $(M, \Om_n, \mu_n, p)$ that  are invariant under
sufficient statistics, which we   reformulate in Proposition \ref{chentsovcampbell} below, see also the explanation that follows Proposition \ref{chentsovcampbell}.  Proposition \ref{congruent1} implies that   such   geometric structures  are preserved under Markov congruent embeddings, which are
easier to understand.
\end{remark} 

 
\begin{proposition}\label{chentsovcampbell}
(1) (cf. \cite[Lemma 11.1 p. 157]{Chentsov1982})  Assume that $C$ is a  continuous function  on  statistical  models $(\Pp_+(\Om_n, \mu_n), \Om_n,\mu_n, Id )$  such that $C$ is invariant under  Markov  congruent embeddings.  Then $C$ is a constant.
 
(2) (cf. \cite[Lemma 11.2, p. 158]{Chentsov1982}) Assume that a  $A$   is
a continuous  1-form field    on  statistical  models 
$(\Pp_+(\Om_n, \mu_n), \Om_n, \mu_n, Id)$   such that $A$ is invariant under  Markov  congruent embeddings.  
Then $A$  equals   zero.  

(3) (cf. \cite[Theorem 11.1, p. 159]{Chentsov1982}) Assume that   $F$ is a continuous  quadratic  form field  on  statistical models  
$(\Pp_+(\Om_n, \mu_n), \Om_n, \mu_n, p)$    such that  $F$ is invariant under  Markov congruent embeddings. Then  

(4) (cf. \cite[Theorem 12.2, p.175]{Chentsov1982}) Assume that  $T$ is
a continuous   covariant   3-tensor field  on
statistical  models  $(\Pp_+(\Om_n, \Om_n), \Om_n, \mu_n, Id)$    such that  $F$ is invariant under  Markov congruent  embeddings.  Then   there is a continuous function  $t: \R \to R$  such that  $T  = t (\sum _{i=1}^n p_i(x))\cdot  T^{AC}$   where $T^{AC}$ is the  Amari-Chentsov  tensor. 
 \end{proposition}
 
The argument of Chentsov for proving (1) (actually for its general form in \cite[Lemma 11.1]{Chentsov1982}) is based on the fact that the elementary  geometry  (with respect to the  Markov congruent embeddings $\Pi$) of
the   spaces  $(\Pp_+ (\Om_n, \mu_n), \Om_m, \mu_m, \Pi)$ is almost homogeneous.  The Chentsov proof of (2) rests  on (1)  and on the
permutation invariance, because a map from $\Om_n$ to itself that
permutes the points of $\Om$ is clearly a sufficient statistic.  The   Chentsov proof of (3)  uses  similar  arguments.
Chentsov   gave  a  proof  of  (4)  in  an equivalent formulation, namely  the uniqueness  of  the Chentsov-Amari  connections  among those affine connections that  are invariant  under  Markov  embeddings, see also Remark \ref{rem:hisac}.
  
  In \cite{Campbell1986} Campell  gave  a generalization of  the second  assertion of Proposition  \ref{chentsovcampbell}  for  parametrized  measure models associated  with   finite sample spaces.  A  generalization  of  Proposition \ref{chentsovcampbell}  for  parametrized  measure models
  associated  with   finite sample spaces is given in the following
  
\begin{corollary}\label{cor:chentsovcampbell}  (1) Assume that a  $A$   is
a continuous  1-form field    on  parametrized  measure  models 
$(\Mm_+(\Om_n, \mu_n), \Om_n, \mu_n, Id)$   such that $A$ is invariant under   Makov congruent  embeddings.  
Then there is a continuous function $c: \R \to \R$  such that for all $x \in M$  and  all $V \in T_x M \subset \R^n$, $A_x(V) = c(\sum_{i=1} ^n p_i(x))\sum_{i =1} ^n p_i \p_V \ln p_i(x)$. 

(2)   Assume that   $F$ is a continuous  quadratic  form field  on  parametrized models  
$(\Mm_+(\Om_n, \mu_n), \Om_n, \mu_n, Id)$    such that  $F$ is invariant under  Markov congruent  embeddings. Then  there are continuous functions $f, d: \R \to \R$ such that   $ F = f (\sum _{i=1}^n p_i(x))\cdot g^F + d(\sum _{i=1}^n p_i(x)) A ^2$ where 
$A$ is  the 1-form field  described  in (1)  with $c = 1$  and $ g^F$ is the Fisher metric. 

(3)  Assume that  $T$ is
a continuous   covariant symmetric  3-tensor field  on
statistical  models  $(\Mm_+(\Om_n,\mu_n), \Om_n, \mu_n, p)$  associated with  finite  sample spaces $\{\Om_n\}$  such that  $F$ is invariant under  Markov congruent  embeddings.  Then   there is a continuous function  $t: \R \to R$  such that  $T  = t (\sum _{i=1}^n p_i(x))\cdot  T^{AC} +  g^F \cdot A_2 + A^3_1 $   where $g^F$ and $T^{AC}$ are the Fisher metric and  the Amari-Chentsov tensor  respectively, and
 $A_1, A_2$  are the fields described in  (1).
\end{corollary}
  
\begin{proof}
(1) Using the  induced  Fisher metric on $T^*\Mm_+(\Om_n, \mu_n)$, we decompose  the 1-form $A\in T^*\Mm_{+}(\Om_k, \mu_k)$
into a sum of two  orthogonal    1-forms $A_0 $ and $ A ^\perp$, where $A_0$ annihilates   the tangent hyperplane 
$T  \Mm^{ p_1+ \cdots +p_k}_{+} (\Om_k, \mu_k) \subset T \Mm_{+} (\Om_k,\mu_k)$  and  $A^\perp = A- A_0$. Since the  Fisher metric is invariant under the Markov  congruent embeddings,   each component $A_0$ and $A ^ \perp$  is also invariant  under the  Markov   congruent embeddings. Taking into account the  first assertion (1)  of
Proposition \ref{chentsovcampbell}, it  follows that $A_0(V)= c(\sum_{i =1}^k p_i  (x))\cdot \sum_{i =1} ^k p_i \p_V\ln p_i $  for some   continuous function $c$. By    the
second  assertion of  Proposition \ref{chentsovcampbell}  the component $A^\perp$ vanishes.  This proves the  the first assertion (1) of  Corollary \ref{cor:chentsovcampbell}.

(2) Using the same argument, i.e.  decomposing $F$ into three  orthogonal  components, we obtain the second assertion of Corollary \ref{cor:chentsovcampbell}    from the 
third assertion of  Proposition \ref{chentsovcampbell} and the  first assertion  of  Corollary \ref{cor:chentsovcampbell}.

(3) The last assertion of Proposition \ref{chentsovcampbell}  is obtained  from its particular case for statistical models  (the last assertion of Propositio \ref{chentsovcampbell})  and   from  the  first  and the second  assertion  of Corollary  \ref{cor:chentsovcampbell}.
\end{proof}


Our proof of the Main Theorem (Theorem \ref{main}) is based on the following  main observation. 
For each step function $\tau$ on $(\Om, \mu)$  subject  to  a statistic $\kappa : (\Om , \mu)\to \Om_n: =\{ E_1, \cdots, E_n\}$ (Definition \ref{step})
there exists a parametrized measure model $(M, \Om, \mu, p)$  and a vector
$V\in T_xM$ such that $p(x) = \mu $ and $\p_V \ln \bar p = \tau$, moreover, $\kappa$ is sufficient with respect to the  parameter $x\in M$  (Lemma \ref{lem:exi}). Thus,  the computation of any pointwise continuous covariant $k$-tensor  field on $\Mm(\Om)$, whose induced $k$-tensor field on parametrized measure models  is invariant under sufficient statistics, is reduced to  the case
$\Om = \Om _n$, which has been  considered by Chentsov for $k = 1,2,3$.



\begin{definition}\label{step} (cf. Example  \ref{ex:suff1}) Let $(\Om, \mu)$ be a finite measure space  and let $\bar \kappa$ be
a decomposition $\Om = D_1 \dot \cup \ldots \dot \cup D_n$ where $D_i$ is measurable.
Denote by $\kappa$ the associated statistic $\Om \to \Om_n, \, \kappa (D_i): = E_i$. 
A function $\tau: \Om \to \R$ is called {\it a step function subject to $\kappa$}, if  $\tau( \om) = \tau_i \cdot \chi_{D_i}(\om)$, where $\tau_i \in \R$  and $\chi _{D_i}$ is the characteristic   function of $D_i$.
\end{definition}

\begin{lemma}\label{lem:exi} 
Let $M = (0,1)$ 
and  $\Om$ be a smooth manifold.
Given a finite measure $\mu \in \Mm(\Omega)$,  a point $x_0 \in M$, and a 
step function $\tau: = \sum _i \tau_i \chi_{D_i}$ on $\Om$  subject to a 
statistic $\kappa: (\Om, \mu) \to \Om_n$, there exist a $k$-integrable 
parametrized measure model $(M, \Om, \mu, p)$ and $V \in T_{x_0} M$ 
such that
\begin{enumerate}
\item $\kappa$ is  sufficient   for the parameter in $M$,
\item $p(x_0) = \mu$,
\item $\p_V \ln \bar p =\sum_i \tau_i \chi_{D_i}$.
\end{enumerate}
\end{lemma}
\begin{proof}
Note that $\kappa$ is a sufficient statistic for a $k$-integrable parametrized measure model 
$(M, \Om, \mu, p)$ iff $p$ is given as in  Definition \ref{suff}, i.e.
\[
\ln \bar p(x, \omega) = \ln \bar p (x, \kappa(\om)) + \ln  t(\omega)= \sum_{i =1}^n s_i(x) \chi_{D_i}(\omega) + \ln  t(\omega)
\]
for smooth functions $s_i: M \to \R$  and $t \in L^1 (\Om)$.  For such $\ln \bar p (x, \om)$ the conditions (2) and (3)   are equivalent to the following
\begin{itemize}
\item $\sum_{i =1}^n s_i(x_0) \chi_{D_i}(\omega) + \ln  t(\omega)= 0$,
\item $\sum_{i =1}^n \p_V s_i(x_0) \chi_{D_i}(\omega) = \sum_i \tau_i \chi_{D_i}(\om)$.
\end{itemize}
Set $t(\om) = 1$. The existence of  functions $s_i (s)$ satisfying the  listed conditions is obvious:  it suffices to choose smooth 
$s_i $ such that $s_i (x_0) = 0$ and $\p_V s_i (x) = \tau_i$. In
  fact, we can simply take $V=\p_x$ and $s_i(x) = (x-x_0) \tau_i $. Finally,
one verifies  that the defined   parametrized  measure model is
$k$-integrable, since the $s_i$ 
are smooth.
\end{proof}

\begin{proof}[Proof of the Main Theorem]  1. Let $A$ be a  pointwise
  continuous  1-tensor field on $\Mm(\Om)$  satisfying the condition
  (1) in the Main Theorem. To  prove  the first assertion of the Main
  Theorem,  it suffices to  assume that $V$ is a step function $\sum_i
  \tau_i \chi_{D_i}$ (using again the identification between the
    tangent vector $V$ and $\p_V \ln \bar p$) subject to a  statistic $\kappa:\Om \to \Om_n$.   By Lemma \ref{lem:exi}  there exists  a $k$-integrable  parametrized measure model $(M, \Om, \mu, p)$
such that 
\begin{enumerate}
\item $\bar p (x, \om) = e ^ {s_i(x)\chi_{D_i}}$ , where $s_i \in C^\infty (M)$, hence $\kappa$ is sufficient for the parameter $x\in M$,
\item $p(x_0) = \mu$,
\item  $\p_V \ln \bar p(x, \om) = \sum_i \tau_i \chi_{D_i}$.
\end{enumerate}
Set 
$$d_i: = \int_{D_i}d\mu.$$
Then $\kappa_*(\mu)  = d_i E_i ^*$, where $E_i^*$ is the Dirac measure concentrated  at $E_i$. Since $A$ is  associated with a statistical field which is invariant under $\kappa _*$ we have
\begin{equation}
A_\mu (\tau) = (A_{\kappa_*(\mu)}(\p_V( \ln  \kappa_*(\bar p))) = A _{ (d_1, \cdots , d_n) } (\tau_1, \cdots, \tau_n)= c (\sum _{i=1}^n d_i)\sum _{i=1} ^n d_i \tau_i,\label{for:A}  
\end{equation}
where $c$  is  the function  defined in Proposition \ref{chentsovcampbell}.2.  Note that 
$$\sum _i d_i = \int_\Om  d\mu, $$
$$\sum_i d_i\tau_i =\sum _i (\int _{D_i} \tau_id\mu) =  \int _\Om  \tau d\mu.$$
This  proves the first assertion in the Main theorem. The next assertions of the Main theorem concerning  specification of  the covariant 1-tensor field $A$ follows immediately.

2. Now assume that $F$ is a pointwise continuous quadratic form on $\Mm(\Om)$ and
$\mu$ is a finite measure.  To prove the second assertion of the Main Theorem   we follow the same line of arguments as above.   It suffices to  prove the validity of the second assertion   for a step function $\tau$   on $\Om$, since $F$ is a quadratic form (otherwise  we have to consider  step functions subjected to different statistics).
We deduce the second assertion of the Main Theorem from  Proposition \ref{chentsovcampbell}.2  using the observation that the Fisher metric on $\Mm(\Om)$ applied to $\tau$
$$g ^F _\mu (\tau) = \int _\Om \tau ^2 d\mu=\sum _{i =1}^n d_i \tau_i ^2$$
is equal to the Fisher metric applied to $\kappa_*(\tau) =  (\tau_1,\cdots, \tau_n)$
$$ g ^ F _{ (d_1, \cdots, d_n)} ([\tau_1,\cdots, \tau_n]) = \sum _{i =1} ^n d_i \tau_i ^2 .$$

3. The last assertion  of the Main Theorem is proven in the same way.   It follows from
Proposition \ref{chentsovcampbell}.2  using the observation that the Amari-Chentsov 3-symmetric tensor on $\Mm(\Om)$ applied to $\tau$
$$T ^{AC} _\mu (\tau) = \int _\Om \tau ^3d\mu=\sum _{i =1}^n d_i \tau_i ^3$$
is equal to the Amari-Chentsov tensor applied to $\kappa_*(\tau) =  (\tau_1,\cdots, \tau_n)$
$$ T ^{AC} _{ (d_1, \cdots, d_n)} ([\tau_1,\cdots, \tau_n]) = \sum _{i =1} ^n d_i \tau_i ^3 .$$

To complete  the proof  of the Main  Theorem  we need to show that
\begin{enumerate}
\item  all the tensor  fields  described in the  Main Theorem  are 
weakly  continuous  on $n$-integrable  parametrized measure models,
\item  the tensor  field $A$ is invariant under  sufficient statistics.
\end{enumerate}
Note that (1)  holds since  the value  $A (V)$ (resp. $F(V)$, $T(V)$) of a tensor field $A$  (resp. $F$, $T$) in the  Main Theorem  at  a continuous  vector field $V$ on $M$ is an  algebraic  function  whose arguments are    tensor fields  of the following forms: $(x, V) \mapsto   c(\int_\Om dp(x))$,
$(x, V)  \mapsto \int _\Om (\p_V\ln \bar p ) ^ k \, d p(x)$, $ k = 1$ (resp $k =2, 3$), which are  continuous 
by the condition (2) of Definition \ref{def:gen}.  
\\
The proof of (2) is similar  to  the proof of Theorem \ref{th:inv}, 
observing that
$$\p _V p(x) = \p_V \ln \bar p(x) \bar p(x)\mu$$ for
$p(x)  = \bar p (x) \mu$ (cf. Remark  \ref{rem:gen1}),  and hence  omitted.
\end{proof}

\begin{remark}\label{rem:ext2}    It is not hard  to prove  a version of the Main Theorem for    local  continuous   statistical covariant tensor fields  on   statistical  models that are
invariant under sufficient  statistics, which is a direct  generalization  of the  Chentsov  theorem \cite{Chentsov1978}, see  its formulation  in Proposition  \ref{chentsovcampbell}.   In particular, it implies the uniqueness  of the  Amari-Chentsov  connections  among   those  affine  connections
on  statistical models  that are invariant  under sufficient  statistics, see also Remark \ref{rem:hisac}. All the  arguments  for the   proof  of  the Main Theorem    also holds  for  this ``statistical" version, since  the image of a statistical  model
under  a  sufficient  statistic is also a statistical model.
\end{remark}

\section{Markov morphisms and sufficient statistics} \label{section-markov}

In this section we introduce the notions of a Markov morphism, a $\mu$-representable Markov morphism,  and a restricted Markov morphism (Definitions \ref{def:markov}, \ref{de:rep},  \ref{equi})  extending the Chentsov   notion of a Markov morphism \cite{Chentsov1965}, and
the notion of a statistical morphism   introduced independently by Morse and Sacksteder in \cite{MS1966}.
These notions are needed for  comparing two statistical models; they stem from the Blackwell concept  of ``comparison of experiments" in \cite{Blackwell1953}.   A novel aspect is  our consideration of a parametrization of the parameter space $M$ of a parametrized measure model 
$(M, \Om, \mu, p)$  as a restricted Markov  morphism (Definition \ref{equi}, 
Example \ref{imstat}). Thus, the geometry of parametrized measure models is intrinsic (Example \ref{imstat}).
We decompose   a Markov morphism  associated with a (positive) Markov transitition kernel as a composition of  a right inverse of a sufficient statistic and a statistic (Theorem \ref{decomp1}).
As a consequence we give a geometric proof of the monotonicity   theory for Markov morphisms  (Corollary \ref{co:monoton}). 
\

{\bf Positivity assumption}. In this section, for the simplicity of the  exposition of the theory, 
when considering Markov transition kernels we restrict ourselves to positive ones.

\ 

\begin{definition}\label{def:markov} (\cite[p. 194]{Chentsov1965}, \cite[p. 205]{MS1966}) 
 {\it A Markov transition from  a measurable space $(\Om, \Sigma)$ to  a measurable space $(\Om ', \Sigma')$}
is a  map  $T: \Om \to \Pp(\Om ', \Sigma')$ such that for each  $S \in \Sigma '$  the function $\int_S d(T(x))$ is a $\Sigma$-measurable  function.
A  Markov transition  $T: \Om \to \Pp(\Om', \Sigma')$ defines  {\it a Markov morphism} 
$T_*: \Mm(\Om, \Sigma) \to \Mm (\Om', \Sigma')$  by
\begin{equation}
T_*( \nu)(S) : = \int_\Om \int _Sd(T (\om)) d\nu \label{for:markov1} 
\end{equation}
for  $S \in \Sigma'$. 
\end{definition}

Since $ T(\Om) \subset \Pp (\Om',\Sigma')$, substituting $S : = \Om'$ in (\ref{for:markov1}), we obtain 
$$T_*( \Mm^a(\Om, \Sigma)) \subset \Mm^a(\Om', \Sigma')$$
for all $a \in \R^+$.  

Next, we assume that $T(\om)$ is   dominated by   a probability  measure $\mu' \in \Pp(\Om', \Sigma')$. Then  there exists a measurable function 
$\Pi _\om : \Om' \to \R$ such that  for all $S \in \Sigma '$ we have
\begin{equation}
T(\om) (S) = \int_S \Pi_{\om}(\om')d\mu'.\label{for:markov2}
\end{equation}
If $ T(\Om) \subset  \Pp(\Om', \mu')$, by (\ref{for:markov2}), there  exists {\it a Markov transition kernel  $\Pi: \Om \times \Om' \to \R$ from  
$\Om$ to $\Mm(\Om', \mu')$} such that
\begin{equation}
T(\om) (S) = \int _S \Pi (\om, \om') d\mu'. \label{for:markovkernel} 
\end{equation}

\begin{definition}\label{de:rep} If (\ref{for:markovkernel}) holds,  $T(\Pi) : = T$ is called  {\it a $\mu'$-representable Markov transition}, 
and $T(\Pi)_*$ is called {\it a $\mu'$-representable Markov morphism}.
\end{definition}

Note that  any Markov transition kernel $\Pi: \Om \times \Om'\to \R$
from $\Om$ to $ \Pp(\Om', \mu')$ 
satisfies
\begin{eqnarray}
   \Pi (\om, \om') \ge 0 \text { for all }  (\om, \om') \in \Om \times \Om',\label{im1} \\
   \int_{\Om'}\Pi (\om, \om')  d\mu'  = 1 \text { for all }  \om\in \Om.\label{im2}
\end{eqnarray} 
 Abbreviate $T(\Pi)_*$ as $\Pi_*$. For any measure
$\nu\in \Mm(\Om)$  and $S \in \Sigma '$ we have 
\begin{equation}
\Pi_* (\nu)(S) = \int_{\Om}\int_S \Pi (\om, \om')d\mu '\,  d\nu.\label{im3}
\end{equation}
It follows
\begin{equation}
\frac{d\Pi_* (\nu)}{d\mu'}(\om') = \int_{\Om} \Pi (\om, \om') d\nu.\label{im3a}
\end{equation}


If  $\Om,\Om'$  are finite sets, then any   Markov morphism $T : \Mm(\Om, \Sigma) \to \Mm (\Om ', \Sigma ')$ is $\mu$-representable   for any  dominant measure $\mu$ on $\Om '$, see   Example \ref{ex1}. This   is not true, if
$\Om, \Om '$ are open domains in $\R^n$, $n \ge 1$, see  the following
 
\begin{example}\label{ex:markov1}  1. (cf. \cite[p. 511]{Chentsov1965})   Let $(\Om, \Sigma)$  be   a  measurable  space. We define  a Markov transition $T^{Id}$ on
$(\Om, \Sigma)$  by setting
$$T ^{Id}(\om) (A) : = \chi_A (\om) \text{ for }  \om \in \Om,$$
where $\chi_A$ is the indicator function of $A\in \Sigma$.  Clearly
$T^{Id}_*$ defines  a Markov  morphism  which is  the identity
transformation of $\Pp  (\Om, \Sigma)$.
Note that    
$T^{Id}_*$  is not a $\mu$-representable Markov morphism for any measure $\mu\in \Mm(\Om, \Sigma)$, if $\Om$ is an open domain in $\R ^n$ with Borel $\sigma$-algebra $\Sigma$,  and $n \ge 1$.  To see this,  we note that
if $\mu$  dominates  all the measures $T^{Id} (\om), \om \in \Om$, then   $\mu$ has no null set, in particular  $\mu ( \{ \om \} ) > 0$  for all  
$\om\in \Om$. It is  easy to see that this is impossible,  since $\dim \Om \ge 1$.

2. Assume that $\kappa: \Om_1\to \Om_2$  is a statistic. Then $\kappa$ defines a Markov transition 
$T^\kappa$ from $(\Om_1, \Sigma_1)$ to $(\Om_2, \Sigma_2)$ by  setting
\begin{equation}
T ^\kappa (\om_1)  (A) : = \chi _A( \kappa(\om_1)) \text { for } \om_1 \in \Om_1
\end{equation}
and $A \in \Sigma_2$.   For $\nu \in \Mm (\Om_1)$  and $S \in \Sigma _2$, using (\ref{for:markov1}), we get
$$T^\kappa _* (\nu)(S) = \int_{\Om_1} \int_S d\chi_A(\kappa(\om_1)) d\nu = \int_{\kappa ^{-1}( S)}d\nu.$$
Hence $T^\kappa _* = \kappa _*$. 
Then $T^{\kappa}_*$ is not
a $\mu_2$-representable  Markov morphism for any $\mu_2 \in \Mm(\Om_2)$, if for instance $\kappa(\Om_1)$ and $\Om_2$ are open domains in 
$\R^n$, $n \ge 1$, since  there exists $\nu\in \Mm(\Om_1)$ such that $\kappa_*(\nu)$ is not dominated by $\mu_2$.   
\end{example}
\medskip

Denote by $C^1(M_1, M_2)$  the  space of all differentiable maps  from  a differentiable manifold $M_1$ to a differentiable manifold $M_2$. Let $(\Om_1, \Sigma_1)$ and  $(\Om_2, \Sigma_2)$ be  measurable spaces.  Denote by ${\mathfrak M}(\Om_1, \Om_2)$ the set of all Markov
morphisms  from  $\Mm(\Om_1)$ to  $\Mm(\Om_2)$.

\begin{definition}\label{equi}   Assume that   $(M_1, \Om_1, \mu_1,
  p_1)$ and $(M_2, \Om_2, \mu_2, p_2)$ are parametrized measure models.
A pair $(f\in  C^1 (M_1, M_2), T \in  {\mathfrak M} (\Om_1, \Om_2))$ is called
{\it a  restricted Markov morphism}, if   for all $x\in M$
\begin{equation}
p_2 (f (x)) = T_* (p_1 (x)).\label{eq:mm}
\end{equation}
\end{definition}


\begin{example}\label{imstat} 1. Assume that $(M, \Om_1, \mu_1, p_1)$
  is a  parametrized measure model and 
$\kappa : \Om_1\to \Om_2$  is a statistic.  Then $(M, \Om_2, \kappa_*(\mu_1), \kappa_* (p_1))$  is a  parametrized measure model.  
By Example \ref{ex:markov1}.2 the pair $(Id, \kappa_*)$ is a Markov morphism. We  also call $(Id, \kappa_*)$ {\it  a  statistic}, if no misunderstanding occurs.

2. Assume that  $(M_2, \Om_2, \mu_2, p_2)$ is a parametrized measure model and  $f: M_1\to M_2$ is a smooth map. Then 
$(M_1, \Om_2, \mu_2,  p_1: = p_2 \circ f)$ is a parametrized measure model and
the pair $(f, Id)$ is a Markov morphism. Such a Markov morphism  is
called {\it generated by a smooth map $f$}. It is easy to see that, if
$f$ is a  differentiable  map, then the Amari-Chentsov structure on $
  M_1$ is obtained from the Amari-Chentsov structure  on $M_2$ via the pull-back map $f^*$. 
\end{example}

\begin{example}\label{ex1}  Let $(\Om_n, \mu_n)$  and $(\Om_m, \mu_m)$ be the  measure spaces in Example \ref{ex:suff1}. Let $\Pi: \Om_n  \times \Om _m \to \R$  be  a mapping  such that
$\Pi_{i,j} : =  \Pi (E_i, F_j)$ satisfies  the  following conditions 
\begin{align}
  \Pi_{i,j} \ge 0 \text{ for all }  1\le i \le n, \, 1\le j \le m,\nonumber\\
\sum_{j=1} ^{m} \Pi_{i,j }= 1 \text{ for all }  1\le i \le n.
\end{align}
Clearly, $\Pi$ is a Markov transition kernel from $\Om _n$ to $\Mm(\Om_m, \mu_m)$. By (\ref{im3})   $\Pi$ induces  a map 
$$\Pi_*: \R^n _{\ge 0} = \Mm(\Om_n, \mu_n) \to \Mm(\Om_m, \mu_m) = \R^m_{\ge 0}, $$
\begin{equation} 
 \Pi_*(E^*_k) (F_j): = \sum_{i = 1} ^{n} \Pi_{i,j} E^*_k(E_i) = \Pi_{kj}.\label{im3f}
\end{equation}
Hence
\begin{equation}
\Pi_* (E^*_k)  = \sum_{j =1}^m\Pi_{kj}F^*_j. \label{for:im3}
\end{equation}
Let 
$$(M_1: = \Pp_+ (\Om_n, \mu_n), \Om_n, \mu_n, p_1(x) : = {x}),  $$
$$  (M_2 : = \Pp_+ (\Om_m, \mu_m), \Om_m, \mu_m, p_1(y): = {y} )$$
be  statistical models. By (\ref{eq:mm}), a pair
$(f \in \Di (M_1, M_2), \Pi \in {\mathfrak M}(\Om_n,\Om_m))$  is a Markov morphism, if and only  if  for all $x\in  M_1$
\begin{equation}
f(x)(F_j) = \Pi_* (x) (F_j)\text{ for all }  1\le j \le m.\label{eq:mark}
\end{equation}
Thus  for  $\Pi \in {\mathfrak M}(\Om_n,\Om_m) $ the pair
$ (f, \Pi)$ is a Markov morphism  if and only if $f = \Pi_*|_{ M_1}$.  We also abbreviate 
$(\Pi_*|_{M_1}, \Pi)$ as $\Pi$ if no misunderstanding occurs.

Next  we  drop the assumption that $n \le m$.  Note that  there  is a canonical map 
$$\chi_n: \Om_n  \to  \Mm (\Om_n, \mu_n),  E_i \mapsto E^* _i.$$
Let  $\kappa: \Om_n\to \Om_m$  be a statistic.  The composition $\chi_m \circ  \kappa: \Om_n \to \Mm(\Om_m, \mu_m)$  defines the following
map $\Pi^\kappa : \Om_n \times \Om_m \to \R$ 
\begin{equation}
\Pi^\kappa (E_i, F_j): = \la \chi_m \circ \kappa (E_i),  F_j\ra
\end{equation}
Clearly  $\sum_{j =1} ^m\Pi^\kappa (E_i, F_j) = 1$  for all $i$. Hence  $\Pi^\kappa$ is a Markov  transition  kernel.  Note that 
$\Pi^\kappa_*:  \Mm(\Om_n, \mu_n) \to \Mm(\Om_m, \mu_m)$ 
coincides with the push-forward map $\kappa_*: \Mm(\Om_n, \mu_n)\to \Mm(\Om_m, \mu_m)$.

\begin{proposition}\label{pro:immark} A linear mapping $\Pi :  \R ^n \to \R ^m   $  is  a Markov congruent  embedding subjected to  a statistic $ \kappa$, if and only 
if  $\Pi^ \kappa_* \circ \Pi   (x) = x $ for all $x \in \R ^n _{\ge 0}$.
A Markov mapping $\Pi: \R^n \to \R^m$ has a  left inverse if and only if it is a  Markov congruent embedding.
 \end{proposition}
 The first assertion of Proposition  \ref{pro:immark} is  obvious. The  second assertion of Proposition \ref{pro:immark} is a reformulation of \cite[Lemma 6.1, p. 77 and Lemma 9.5, p.136]{Chentsov1982}.  
\end{example}

 Let  $(M, \Om_1,\mu_1, p_1)$ be a parametrized measure model, 
 $(\Om_2, \mu_2)$ a probability space and $\Pi:\Om_1 \times \Om_2 \to \R$  a Markov transition kernel  from $\Om_1$ to
 $\Pp(\Om_2, \mu_2)$. 
 We define a function $\Pi ^{[p_1]}: M \times \Om_1 \times \Om_2 \to \R$ by setting:
\begin{equation}
\Pi^{[p_1]}  (x, \om_1, \om_2) : = \Pi(\om_1, \om_2) \bar{p}_1(x, \om_1). \label{pila1}
\end{equation}
Using (\ref{im2}),  we get  for all $x \in M$  and  any measurable  set $S\subset \Om_1$
\begin{equation}
\int_{S \times \Om_2} \Pi ^{[p_1]}  (x, \om_1, \om_2)\mu_1 \mu_2 = \int_{S} \bar{p}_1(x, \om_1) d\mu_1.\label{for:markov4}
\end{equation}

\begin{lemma}\label{lah1}  Then $ (M, \Om_1 \times \Om_2, \mu_1 \mu_2, \Pi^{[p_1]}(x, \om_1, \om_2) d\mu_1d\mu_2)$
is a parametrized  measure model.  Moreover, the Amari-Chentsov structure on $(M, \Om_1\times \Om_2,\mu_1\mu_2,\\
\Pi^{[p_1]}(x, \om_1, \om_2)d\mu_1d\mu_2)$ coincides
with the Amari-Chentsov structure on $(M, \Om_1, \mu_1,p_1)$.
\end{lemma}

\begin{proof}
Let $\pi_1: \Om_1 \times \Om_2 \to \Om_1$ be a projection. Since $\mu_2$ is a probability  measure, $(\pi_1)_*(\mu_1\mu_2) = \mu_1$. Comparing (\ref{suff1}) with (\ref{pila1}), we observe that $\pi_1$ is a sufficient statistic with respect to the parameter   $ x\in (M, \Om_1 \times \Om_2, \mu_1 \mu_2,\\
  \Pi^{[p]}(x,\om_1, \om_2)d\mu_1d\mu_2)$.  By (\ref{for:markov4}), the parametrized measure model $(M,\Om_1, \mu_1,p_1)$ is the image of  $(M, \Om_1 \times \Om_2, \mu_1 \mu_2, \Pi^{[p]}(x,\om_1, \om_2)d\mu_1d\mu_2)$  under the Markov morphism $(Id, (\pi_1)_*)$.
Combining this with Lemma  \ref{fisherneyman},  we obtain immediately  Lemma \ref{lah1}.
\end{proof}

We obtain immediately from the  proof of Lemma \ref{lah1}

\begin{corollary}\label{lah2} Let $(M, \Om_1, \mu_1, p_1)$ be a
  parametrized measure model  and $\mu_2$  a probability measure on $\Om_2$. The projection 
$\pi_1 : \Om_1 \times \Om_2\to \Om_1$ is a  sufficient statistic for  the  parametrized measure model 
$(M, \Om_1 \times \Om_2,\mu_1\mu_2, \Pi^{ [p]}(x,\om_1, \om_2) d\mu_1d\mu_2
)$. 
\end{corollary}

Next, we consider  a decomposition of  a restricted Markov morphism.

\begin{theorem}\label{decomp1} Let $(Id, \Pi_*):(M_1, \Om_1, \mu_1, p_1) \to (M_1, \Om_2, \mu_2, p_2)$ be a restricted Markov morphism  between statistical  models, where 
$\Pi_*$ is $\mu_2$-representable by a positive Markov kernel. Then $(Id, \Pi_*)$ is a composition of the inverse of  a Markov morphism, associated with a  sufficient statistic, and a statistic.
\end{theorem}

\begin{proof} Let  $\pi_2: \Om_1 \times \Om_2 \to \Om_2$ be the projection onto the second factor.  
Then   for  any $x\in M$  and any measurable set $S \subset \Om_2$ we have
\begin{eqnarray}
\int_S (\pi_2)_* [(\Pi^{[p_1]}(x, \om_1, \om_2) d\mu_1 d\mu_2] = \int_{\pi_2 ^{-1} (S)}\Pi^{[p_1]}(x, \om_1, \om_2) d\mu_1 d\mu_2=\nonumber\\
= \int_{\Om_1} \int_S \Pi(\om_1, \om_2) \bar{p}_1(x, \om_1)d\mu_1d\mu_2  = \Pi_* (p_1(x)) [S]. \label{eq:id1}
\end{eqnarray}
Let $(Id, \Pi_{1,12}): (M_1, \Om_1, \mu_1, p_1) \to (M_1, \Om_1 \times \Om_2, \mu_1\mu_2, \Pi^{[p_1]}(x, \om_1, \om_2)d\mu_1d\mu_2)$  be a map between statistical  models defined by
$$(x,\Om_1, \mu_1, p_1(x)) \mapsto  (x,\Om_1\times \Om_2, \mu_1\mu_2, \Pi^{[p_1]}(x,\om_1, \om_2)d\mu_1d\mu_2).$$
 Then $(Id,\Pi_{1,12})$ is the inverse of  the  Markov morphism $(Id, (\pi_1)_*)$, associated with a sufficient statistic   by Corollary \ref{lah2}. 
 By (\ref{eq:id1}),  $(Id, \Pi_*)$ is   a composition of  $(Id,\Pi_{1,12})$ with $(Id, (\pi_2)_*)$. This completes the proof of  Theorem \ref{decomp1}.
\end{proof}

Let $M_1 = M_2$. A  restricted Markov morphism of form $(f, T_*)$  is called {\it representable} if $f$ is a diffeomorphism, and
$T_*$ is $\mu$-representable. 

\begin{corollary}\label{co:monoton} (cf. \cite[p. 31]{AN2000}) 1.  Representable restricted Markov morphisms decrease the Fisher metric   on  $k$-integrable statistical  models $(M, \Om, \mu, p)$  where $\Om$ is a smooth manifold and  $\mu$ is  a Lebesgue  probability measure.

2. The Fisher metric is the unique  up to a constant weakly continuous quadratic 2-form field   on statistical  models  associated with finite sample spaces $\{\Om_n\}$ that is monotone  under  representable restricted Markov  morphisms.
\end{corollary}

\begin{proof} The first assertion of Corollary \ref{co:monoton} is an immediate consequence of Theorem \ref{decomp1} and Theorem \ref{th:monoton}, noting that $\pi_2 $ is smooth.

The second  assertion of Corollary \ref{co:monoton} is  a consequence  of Theorem \ref{decomp1} and  Proposition \ref{chentsovcampbell}, taking into account the following fact.   A congruent Markov embedding $\Pi: \Pp (\Om_n, \mu_n) \to 
\Pp(\Om_m, \mu_m)$ subjected to  a statistic $\kappa$   satisfies  $\Pi ^\kappa _* \Pi (x) = Id$ by Proposition
\ref{pro:immark}. Since any quadratic form field on $( \Mm_+(\Om_n, \mu_n), \mu_n,  p (x) : = x)$
that is monotone under  Markov morphisms  is monotone under
Markov congruent embeddings, it follows that such a quadratic form is invariant under sufficient statistics $\kappa_*$
and also invariant under Markov congruent embeddings.  Chentsov's result implies that such a quadratic form  is the Fisher metric up to a constant.
\end{proof}


\section{The Pistone-Sempi structure}\label{pistone}

In this section we study  the relations between  $k$-integrable parametrized measure models and statistical models in the Pistone-Sempi theory. First,  we show that the Pistone-Sempi manifold is a $k$-integrable  parametrized measure model for any $k$ (Proposition \ref{pro:pise}).
We also construct an example of a $k$-integrable parametrized measure model which does not admit a continuous map into the 
space $\Mm_+(\Omega, \mu_0)$ with the topology of Pistone and Sempi (Example \ref{example-notPS}).

In Section \ref{pmm}, we considered the $L^1$-topology of $\Mm_+(\Omega,\mu_0)$. However, 
this set carries also a stronger natural topology, discovered by Pistone and Sempi, which is referred to as 
{\it the exponential topology} (also {\it $e$-topology}) \cite[\S 2.1]{PS1995}. 
In fact, Pistone and Sempi considered only  the space $\Pp_{+} (\Om, \mu)$  but  their theory  works also for
$\Mm_{+} (\Om, \mu) = \Pp_{+} (\Om, \mu) \times \R^+$.  
Let us briefly recall the notion of the  $e$-topology, which  is defined using the
notion of convergence of sequences.

\begin{definition}\label{def:sempi1}\cite[Definition 1.1]{PS1995} The sequence $(\mu_n)_{n \in \N}$ in $\Mm_{+} (\Om, \mu)$ is $e$-convergent (exponentially convergent) to $\mu$ if $(\mu_n)_{n\in \N}$ tends to $\mu$ in the $L^1$-topology as $n \to \infty$, and, moreover, the sequences 
$(d\mu_n/d\mu)_{n\in \N}$ and $(d\mu/d\mu_n)_{n\in \N}$ are eventually bounded in each $L^p(\Omega,\mu),\, p > 1$, that is, $d\mu_n/d\mu$ and 
$d\mu/d\mu_n$  converge to 1 with respect to all $p$-seminorms 
$L^p (\Omega,\mu),\, p > 1$.
\end{definition}

While $\Mm_+(\Omega,\mu_0)$ is connected with respect to the $L^1$-topology, its set of connected components with respect to   
the $e$-topology is more interesting. In what follows we briefly describe these components and their structure. Although the stated 
facts are known from the work of Pistone an Sempi, our presentation is slightly different and illuminates more abstract aspects.              

\subsection{Orlicz spaces} \label{section:Orlicz}

In this section, we briefly recall the theory of Orlicz spaces which is needed in section \ref{pistonesempi} for the description of 
the geometric structure on $\Mm(\Omega)$. Most of the results can be found e.g. in \cite{KR1958}.

A function $\phi: \R \to \R$ is called a {\em Young function } if $\phi(0) = 0$, $\phi$ is even, convex, strictly increasing on $[0, \infty)$ and $\lim_{t \to \infty} t^{-1} \phi(t) = \infty$. Given a finite measure space $(\Om, \mu)$ and a Young function $\phi$, we define the {\em Orlicz space}
\[
L^\phi(\mu) := \left\{ f : \Om \to \R \mid \int_\Om \phi\left(\frac f a\right) d\mu < \infty \mbox{ for some $a > 0$}\right\},
\]
and on $L^\phi(\mu)$ we define the {\em Orlicz norm }
\[
||f||_{\phi, \mu} := \inf \left\{ a > 0 \mid \int_\Om \phi\left(\frac f a\right) d\mu \leq 1\right\}.
\]
For any Young function, $(L^\phi(\mu), || \cdot ||_{\phi, \mu})$ is a Banach space. Moreover, a sequence $(f_n)_{n \in \N} \in L^\phi(\mu)$ converges to $0$ if and only if
\[
\lim_{n \to \infty} \int_\Om \phi (p f_n)\ d\mu = 0 \qquad \mbox{ for all $p > 0$}.
\]

\begin{proposition} \label{prop:compare-Orlicz}
Let $(\Om, \mu)$ be a finite measure space, and let $\phi_1, \phi_2: \R \to \R$ be two Young functions. If
\[
\limsup_{t \to \infty} \frac{\phi_1(t)}{\phi_2(t)} < \infty,
\]
then $L^{\phi_2}(\mu) \subset L^{\phi_1}(\mu)$, and the inclusion is continuous, i.e., $||f||_{\phi_1, \mu} \leq c\ ||f||_{\phi_2, \mu}$ for some $c > 0$ and all $f \in L^{\phi_2}(\mu)$. In particular, if
\[
0 < \liminf_{t \to \infty} \frac{\phi_1(t)}{\phi_2(t)} \le \limsup_{t \to \infty} \frac{\phi_1(t)}{\phi_2(t)} < \infty,
\]
then $L^{\phi_1}(\mu) = L^{\phi_2}(\mu)$, and the Orlicz norms $|| \cdot ||_{\phi_1, \mu}$ and $|| \cdot ||_{\phi_2, \mu}$ are equivalent.
\end{proposition}

\begin{proof}
By our hypothesis, $\phi_1(t) \leq K \phi_2(t)$ for some $K \ge 1$ and all $t \ge t_0$. Let $f \in L^{\phi_2}(\mu)$ and $a > || f ||_{\phi_2, \mu}$. Moreover, decompose
\[
\Om := \Om_1 \dot \cup \Omega_2 \qquad \mbox{with} \qquad \Omega_1:= \{ \omega \in \Omega \mid |f(\omega)| \ge a t_0 \}.
\]
Then
\begin{eqnarray*}
K & \ge & K \int_\Omega \phi_2 \left( \frac{|f|}a \right) d\mu \ge \int_{\Omega_1} K \phi_2 \left( \frac{|f|}a \right) d\mu\\
& \ge & \int_{\Omega_1} \phi_1 \left( \frac{|f|}a \right) d\mu \qquad \qquad \qquad \mbox{as $\dfrac{|f|}a \ge t_0$ on $\Omega_1$}\\
& = & \int_\Omega \phi_1 \left( \frac{|f|}a \right) d\mu - \int_{\Omega_2} \phi_1 \left( \frac{|f|}a \right) d\mu \\
& \ge & \int_\Omega \phi_1 \left( \frac{|f|}a \right) d\mu - \int_{\Omega_2} \phi_1( t_0) d\mu \qquad \mbox{as $\dfrac{|f|}a < t_0$ on $\Omega_2$}\\
& \ge & \int_\Omega \phi_1 \left( \frac{|f|}a \right) d\mu - \phi_1( t_0) \mu(\Omega).
\end{eqnarray*}

Thus, $\int_\Omega \phi_1 \left( \frac{|f|}a \right) \le K + \phi_1( t_0) \mu(\Omega) =: c$, hence $f \in L^{\phi_1}(\mu)$. Convexity and $\phi_1(0) = 0$ implies that $\phi_1(c^{-1} t) \le c^{-1} \phi_1(t)$, as $c > 1$ and hence,
\[
\int_\Omega \phi_1 \left( \frac{|f|}{ac}  \right)  d\mu \le c^{-1} \int_\Omega \phi_1 \left( \frac{|f|}a  \right) d\mu \le 1,
\]
so that $a c \ge || f ||_{\phi_1, \mu}$ whenever $a > || f ||_{\phi_2, \mu}$, and this shows the claim.
\end{proof}

The following lemma is a straightforward consequence of the definitions and we omit the proof.

\begin{lemma} \label{lem:young-stretch}
Let $(\Omega, \mu)$ be a finite measure space, let $\phi: \R \to \R$ be a Young function, and let $\tilde \phi(t) := \phi(\lambda t)$ for some constant $\lambda > 0$.

Then $\tilde \phi$ is also a Young function. Moreover, $L^\phi (\mu) = L^{\tilde \phi}(\mu)$ and $|| \cdot ||_{\tilde \phi, \mu} = \lambda || \cdot ||_{\phi, \mu}$, so that these norms are equivalent.
\end{lemma}

Furthermore, we investigate how the Orlicz spaces relate when changing the measure $\mu$ to an equivalent measure $\mu' \in \Mm(\Omega, \mu)$.

\begin{proposition} \label{prop:orlitz-changemeasure}
Let $0 \neq \mu' \in \Mm(\Om, \mu)$ be a measure such that $d\mu'/d\mu \in L^p(\Om, \mu)$ for some $p > 1$, and let $q > 1$ be the dual index, i.e., $p^{-1} + q^{-1} = 1$. Then for any Young function $\phi$ we have
\[
L^{\phi^q}(\mu) \subset L^\phi(\mu'),
\]
and this embedding is continuous.
\end{proposition}

\begin{proof}
Let $h := d\mu'/d\mu \in L^p(\Om, \mu)$ and $c := ||h||_p > 0$. If $f \in L^{\phi^q}(\mu)$ and $a > ||f||_{\phi^q, \mu}$, then by H\"older's inequality we have
\[
\int_\Om \phi\left( \frac{|f|}a \right)\ d\mu' = \int_\Om \phi\left( \frac{|f|}a \right) h\ d\mu \le c \left| \left|\phi\left( \frac{|f|}a \right)\right|\right|_q = c \underbrace{\left| \left|\phi^q \left( \frac{|f|}a \right)\right|\right|_1^{1/q}}_{\le 1} \le c.
\]
Thus, $f \in L^\phi(\mu')$, and $a \ge ||f||_{c^{-1} \phi, \mu'}$ whenever $a > ||f||_{\phi^q, \mu}$, hence $||f||_{\phi^q, \mu} \ge ||f||_{c^{-1} \phi, \mu'}$. This shows the claim as $|| \cdot ||_{c^{-1} \phi, \mu'}$ and $|| \cdot ||_{\phi, \mu'}$ are equivalent norms on $L^\phi(\mu')$ by Proposition \ref{prop:compare-Orlicz}.
\end{proof}

\subsection{Exponential tangent spaces} \label{pistonesempi}

For an arbitrary $\mu \in \Mm_+ (\Omega, \mu_0)$, we define the set
\[
\hat{B}_{\mu}(\Omega) := \{ f :  \Omega \to [- \infty, + \infty] \; : \; e^f \in L^1(\Omega, \mu) \} ,
\]
which by H\"older's inequality is a convex cone inside the space of measurable functions $\Omega \to [ - \infty, + \infty ]$. For $\mu_0$, 
there is a bijection 
\[
    \log_{\mu_0} : \Mm_+(\Omega, \mu_0) \to \hat{B}_{\mu_0}(\Omega), \qquad \phi \, \mu_0 \mapsto \log (\phi),
\]
and for $\mu_0' \in \Mm_+(\Omega, \mu_0)$ we have $\log_{\mu_0'} = \log_{\mu_0} - u$ where $u := \log_{\mu_0'} (\mu_0')$. That is, $\log_{\mu_0}$
canonically identifies $\Mm_+(\Omega, \mu_0)$ with a convex set. Moreover, we let
\begin{eqnarray*}
B_\mu(\Om) & := & \hat B_\mu(\Om) \cap (-\hat B_\mu(\Om))\\
& = & \{ f: \Om \to [-\infty, \infty] \mid e^{\pm f} \in L^1(\Om, \mu)\}\\
& = & \{ f: \Om \to [-\infty, \infty] \mid e^{|f|} \in L^1(\Om, \mu)\}
\end{eqnarray*}
and
\[
B_\mu^0(\Om) := \{ f \in B_\mu(\Om) \mid (1+s) f \in B_\mu(\Om) \mbox{ for some $s > 0$}\}.
\]
The points of $B_\mu^0(\Om)$ are called {\em inner points} of $B_\mu(\Om)$.

Note that  for $\mu \in \Mm_+(\Om, \mu_0)$  we have $B_\mu (\Om)  \subset B_{\mu_0} (\Om)$.\\

\begin{definition}\label{def:tangent}
Let $\mu \in \Mm_+(\Om, \mu_0)$. Then
$$
T_\mu \Mm_+(\Om, \mu_0) := \{ f: \Om \to [-\infty, \infty] \mid \mbox{ $t f \in B_{\mu} (\Om)$ for some $t \neq 0$}\}
$$
is called the {\it exponential tangent space of $\Mm_+(\Omega, \mu_0)$ at $\mu$}.
\end{definition}


Evidently, this space coincides with the Orlicz space $T_\mu \Mm_+(\Om, \mu_0) = L^{\cosh t - 1}(\mu)$ and hence has a Banach norm. Moreover, $B_\mu(\Om) \subset T_\mu \Mm_+(\Om, \mu_0)$ contains the unit ball w.r.t. the Orlicz norm and hence is a neighborhood of the origin. Furthermore, $\lim_{t \to \infty} t^p/(\cosh t - 1) = 0$ for all $p \ge 1$, so that Proposition \ref{prop:compare-Orlicz} implies that

\begin{equation} \label{eq:inclusion}
    L^\infty (\Omega, \mu_0) \; \subset \; T_\mu \Mm_+ (\Omega, \mu_0) \; \subset \; \bigcap_{p \geq 1} L^p (\Omega, \mu),
\end{equation} 
where all inclusions are continuous.
\begin{remark}
In \cite[Definition 6]{GP1998}, $T_\mu \Mm_+(\Om, \mu_0)$ is called the {\em Cramer class }�of $\mu$. Moreover, in \cite[Proposition 7]{GP1998} and \cite[Definition 2.2]{PS1995}, the subspace of {\em centered Cramer class } is defined as the functions $u \in T_\mu \Mm_+(\Om, \mu_0)$ with $\int_\Om u\ d\mu = 0$. Thus, the space of centered Cramer classes is a closed subspace of codimension one.
\end{remark}


In order to understand the topological structure of $\Mm_+(\Om, \mu_0)$ with respect to the $e$-topology, 
it is useful to introduce the following preorder on 
$\Mm_+(\Om, \mu_0)$:
\begin{equation} \label{def:partial-order}
\mu' \preceq \mu \quad \mbox{if and only if } \quad \mu' = \phi \mu 
\mbox{ with }
\phi \in L^{p}(\Om, \mu)
\mbox{ for some $p > 1$}.
\end{equation}

In order to see that $\preceq$ is indeed a preorder, we have to show transitivity, as the reflexivity of $\preceq$ is obvious. 
Thus, let $\mu'' \preceq \mu'$ and $\mu' \preceq \mu$, so that  $\mu' = \phi \mu$ and $\mu'' = \psi \mu'$ with $\phi \in L^p(\Om, \mu)$ and $\psi \in L^{p'}(\Om, \mu')$, then $\phi^p, \psi^{p'} \phi \in L^1(\Om, \mu)$ for some $p, p' > 1$. Let $\lambda := (p'-1)/(p + p'-1) \in (0,1)$. Then by H\"older's inequality, we have:
\[
L^1(\Om, \mu_1) \ni (\psi^{p'} \phi)^{1-\lambda} (\phi^p)^\lambda = \psi^{p'(1-\lambda)} \phi^{1+ \lambda (p-1)} = (\psi \phi)^{p''},
\]
where $p'' = p p'/(p + p' -1) > 1$, so that $\psi \phi \in L^{p''}(\Om, \mu)$, and hence, $\mu'' \preceq \mu$ as $\mu'' = \psi \phi \mu$.

From the preorder $\preceq$ we define the equivalence relation on $\Mm_+(\Om, \mu_0)$ by
\begin{equation} \label{df:equivalence}
\mu' \sim \mu \qquad \mbox{if and only if $\mu' \preceq \mu$ and $\mu \preceq \mu'$},
\end{equation}
in which case we call $\mu$ and $\mu'$ {\em similar}, and hence we obtain a partial ordering on the set of equivalence classes 
$\Mm_+(\Om, \mu_0)/_\sim$
\[
{}[\mu'] \preceq [\mu] \qquad \mbox{if and only if} \qquad \mu' \preceq \mu.
\]

If $\mu' \preceq \mu$, then $T_\mu \Mm_+ (\Om, \mu_0) \subset T_{\mu'} \Mm_+ (\Om, \mu_0)$ is continuously embedded. Namely, $\lim_{t \to \infty} (\cosh t - 1)^q/(\cosh (qt) - 1) = 2^{1-q}$, and then we apply Propositions \ref{prop:compare-Orlicz} and \ref{prop:orlitz-changemeasure} as well as Lemma \ref{lem:young-stretch}. 

%
%
%
%

In particular, if $\mu \sim \mu'$ then $T_\mu \Mm_+(\Om, \mu_0) = T_{\mu'} \Mm_+(\Om, \mu_0)$, 
and this space we denote by $T_{[\mu]} \Mm_+(\Om, \mu_0)$. This space is therefore  equipped with a family of equivalent Banach norms, and we have continuous inclusions
\begin{equation} \label{eq:inclusion-tangentspace}
T_{[\mu']} \Mm_+(\Om, \mu_0) \supset T_{[\mu]} \Mm_+(\Om, \mu_0) \qquad \mbox{if} \qquad [\mu'] \preceq [\mu].
\end{equation}

\begin{remark} In general, the subspace in (\ref{eq:inclusion-tangentspace}) will be neither closed nor dense. Indeed, it is not hard to show that $f \in T_{[\mu']} \Mm_+(\Om, \mu_0)$ lies in the closure of $T_{[\mu]} \Mm_+(\Om, \mu_0)$ if and only if
\[
(|f| + \epsilon \log (d\mu'/d\mu))_+ \in T_{[\mu]} \Mm_+(\Om, \mu_0) \qquad \mbox{for all $\epsilon > 0$}.
\]
\end{remark}




The following now is a reformulation of Propositions 3.4 and 3.5 in \cite{PS1995}.

\begin{proposition}\label{prop:econv} A sequence $(g_n)_{n \in \N} \in \Mm(\Om, \mu_0)$ is $e$-convergent to $g \in \Mm(\Om, \mu_0)$ if and only if $g_n \mu_0 \sim g \mu_0$ for large $n$, and $u_n := \log |g_n| \in T_{g \mu_0}\Mm_+(\Om, \mu_0)$ converges to $u_0 := \log |g| \in T_{g \mu_0}\Mm_+(\Om, \mu_0)$ in the Banach norm on $T_{g \mu_0}\Mm_+(\Om, \mu_0)$ described above.
\end{proposition}

By virtue of this proposition, we shall refer to the topology on $T_\mu\Mm_+(\Om, \mu_0)$ obtained above as the {\em topology of $e$-convergence} or the {\em $e$-topology}. Our description allows us to describe in a different way the Banach manifold structure on $\Mm(\Om, \mu_0)$ defined in \cite{PS1995}.

\begin{theorem} \label{thm:PS-components}
Let $K \subset \Mm_+(\Om, \mu_0)$ be an equivalence class w.r.t. $\sim$, and let 
$T := T_{[\mu]} \Mm_+(\Om, \mu_0)$ for $\mu \in K$ be the common exponential tangent space, equipped with the $e$-topology. Then for all $\mu \in K$,
\[
A_\mu := \log_\mu(K) \subset T
\]
is open convex. In particular, the identification $\log_\mu: A_\mu \to K$ allows us to canonically identify $K$ with a open convex subset of the affine space associated to $T$.
\end{theorem}

\begin{remark} This theorem shows that the equivalence classes w.r.t. $\sim$ are the connected components of the $e$-topology on $\Mm(\Om, \mu_0)$, and since each such component is canonically identified as a subset of an affine space whose underlying vector space is equipped with a family of equivalent Banach norms, it follows that $\Mm(\Om, \mu_0)$ is a Banach manifold. This is the affine Banach manifold structure on $\Mm(\Om, \mu_0)$ described in \cite{PS1995}, therefore we refer to it as the {\em Pistone-Sempi structure}.
\end{remark}

\begin{proof} {\em (Theorem \ref{thm:PS-components})} If $f \in A_\mu$, then, by definition, $(1+s) f, -s f \in \hat B_\mu(\Om)$ for some $s > 0$. In particular, $s f \in B_\mu(\Om)$, so that $f \in T$ and hence, $A_\mu \subset T$. Moreover, if $f \in A_\mu$ then $\lambda f \in A_\mu$ for $\lambda \in [0,1]$.

Next, if $g \in A_\mu$, then $\mu' := e^g \mu \in K$. Therefore, $f \in A_{\mu'}$ if and only if $K \ni e^f \mu' = e^{f+g} \mu$ if and only if $f+g \in A_\mu$, so that $A_{\mu'} = g + A_\mu$ for a fixed $g \in T$. From this, the convexity of $A$ follows.

Therefore, in order to show that $A_\mu \subset T$ is open, it suffices to show that $0 \in A_{\mu'}$ is an inner point for all $\mu' \in K$. For this, observe that for $f \in B_{\mu'}^0(\Om)$ we have $(1+s) f \in B_{\mu'}(\Om)$ and hence $e^{\pm (1+s) f} \in L^1(\Om, \mu')$, so that $e^f \in L^{1+s}(\Om, \mu')$ and $ e^{-f} \in L^{1+s}(\Om, \mu') \subset L^s(\Om, \mu')$, whence $e^f \mu' \sim \mu' \sim \mu$, so that $e^f \mu' \in K$ and hence, $f \in A_{\mu'}$. Thus, $0 \in B_{\mu'}^0(\Om) \subset A_{\mu'}$, and since $B_{\mu'}^0(\Om)$ contains the unit ball of the Orlicz norm, the claim follows.
\end{proof}

In the terminology which we developed, we can formulate the significance of the Pistone-Sempi structure on $\Mm_+(\Om, \mu_0)$ as follows.

\begin{proposition}\label{pro:pise}  The quadruple $(\Mm_+(\Om, \mu), \Om, \mu, i_{can})$ is a $k$-integrable  statistical model for all $k \ge 1$.
\end{proposition}

\begin{proof} Note that  for
$x \in \Mm_+(\Om, \mu)$ we have   $\ln \bar p(x, \om) = \ln  x(\om)$.  Using this and the definition of the Pistone-Sempi manifold, we conclude  that
the first condition in Definition
\ref{def:gen} holds for the Pistone-Sempi  manifold.  The second condition in Definition \ref{def:gen} also holds for the Pistone-Sempi manifolds, since  by Theorem \ref{thm:PS-components} for $f \in T_x \Mm_+ (\Om, \mu)$ we have $\p_f \ln x( \om)  = f(\om)$ and by (\ref{eq:inclusion}) the inclusion $T_\mu \Mm_+(\Om, \mu) \to L^k (\Om, \mu)$ is continuous for all $k\ge 1$.
The $e$-continuity of $|\p_{f_1}(\om)\cdots \p_{f_k}(\om)|$ in $x$ holds  obviously, since   in  coordinates  $A_\mu$ that expression  does  not depend  on $x$.
\end{proof}

The following example shows that the notion of a $k$-integrable parametrized measure model is more general than the corresponding
notion within the theory of Pistone and Sempi.

\begin{example} \label{example-notPS}
Let $\Om := (0,1)$, and consider the $1$-parameter family of finite measures
\[
 p(x) \; := \; \bar{p}(x, t) \, dt \; :=  \; \exp \left(  - \frac{x^2}{t^{\frac{1}{k}}}  \right) d t \; \in \; \Mm_+((0,1), dt) , \qquad x \in {\mathbb R}.
\]
This family defines a $(k-1)$-integrable parametrized measure model: Consider the map 
\[
   \ln \bar{p}(\cdot, t): \; x \mapsto \; - \frac{x^2}{t^{\frac{1}{k}}} . 
\]
It is continuously differentiable for all $t \in (0,1)$ and therefore satisfies condition (1) of Definition \ref{def:gen}. 
Now we come to condition (2): With a continuous vector field $V: {\Bbb R} \to {\Bbb R}$, we have
\[
         \partial_V \ln \bar{p}(x, t) \; = \; V(x) \, \frac{\partial}{\partial x} \, \ln \bar{p}(x,t) \; = \;  
         V(x) \, \frac{\partial}{\partial x}  \left(- \frac{x^2}{t^{\frac{1}{k}}} \right) \; = \; 
         - V(x) \,  \frac{2 \, x}{t^{\frac{1}{k}}} .
\]
We now show that the function $t \mapsto \partial_V \ln \bar{p}(x, t)$ belongs to $L^{j}((0,1), p(x))$ for all $j \leq k - 1$: 
\begin{eqnarray*}
  I^{(j)}(x) & := & {\| \partial_V \ln \bar{p}(x, t) \|}_{L^j((0,1), p(x))}^j \\
   & = & \int_0^1 \left( \frac{2 \, | x \, V(x) |}{t^{\frac{1}{k}}} \right)^j \exp \left( - \frac{x^2}{t^{\frac{1}{k}}} \right) \, dt  \\
   & \leq &  \frac{\big( 2 \, | x \, V(x) | \big)^j}{e} \int_0^1 \frac{1}{t^{\frac{j}{k}}} \, dt  \\
   & = &  \frac{\big( 2 \, | x \, V(x) | \big)^j}{e} \, \frac{k}{k - j} \\
   & < & \infty .
\end{eqnarray*}
Finally, since $\dim M = 1$,  it suffices to show that the function $x \mapsto I^{(j)}(x)$ is continuous. In order to verify the 
continuity in a point $x_0 \in {\Bbb R}$ it is sufficient to consider the restriction of $I^{(j)}$ to the closed interval 
$[x_0 - \varepsilon, x_0 + \varepsilon]$ with some positive number $\varepsilon$. On this interval, the corresponding integrand is 
upper bounded by a function that only depends on $t$ and is integrable: 
\[
     \left( \frac{2 \, | x \, V(x) |}{t^{\frac{1}{k}}} \right)^j \exp \left( - \frac{x^2}{t^{\frac{1}{k}}} \right) \; \leq \;  
     \frac{c}{t^{\frac{j}{k}}}, \qquad c \geq 0.
\]
Therefore, by the continuity lemma for integrals, $I^{(j)}$ is continuous, which completes the proof that our family is $(k-1)$-integrable
parametrized measure model. However, it does not define a model in the sense of Pistone and Sempi. In order to see this we show that 
for all $x \neq 0$, $p(x)$ and $p(0)$ are not similar: Obviously,  
\[
    dt \; = \; \exp \left( { \frac{x^2}{t^{\frac{1}{k}}} } \right) dp(x) .
\]
The similarity of $dp(x)$ and $dt$ would imply that $\frac{dt}{dp (x)}$ is in $L^{1+s}((0,1), dp(x))$ for some $s > 0$ (see \ref{def:partial-order}
and \ref{df:equivalence}). However, for all $s > 0$, we have
\begin{eqnarray*}
 \int_0^1 \left(\exp \left(  \frac{x^2}{t^{\frac{1}{k}}}  \right)  \right)^{1+s} dp(x) 
    & = & \int_0^1 \exp \left( \frac{s \, x^2}{t^{\frac{1}{k}}} \right) dt \\
    & \geq &  \int_0^1 \frac{1}{k!} {\left( \frac{s \, x^2}{t^{\frac{1}{k}}} \right)}^k   \, dt \\
    & = & \infty.
\end{eqnarray*}   
Thus, $p(x)$ and $dt$ are in different $e$-connected components of $\Mm_+((0,1), dt)$ and, therefore, the map $p$ cannot be continuous 
with respect to the $e$-topology. Hence, the parametrized measure model 
cannot be considered as a submanifold of $\Mm_+((0,1), dt)$ in the sense of Pistone and Sempi.
\end{example}

We end this section with the following result which illustrates how the ordering $\preceq$ provides a stratification of $\hat B_{\mu_0}(\Om)$.

\begin{proposition} \label{prop:stratification}
Let $\mu_0', \mu_1' \in \Mm_+ (\Om, \mu_0)$ with $f_i := \log_{\mu_0}(\mu_i') \in \hat B_{\mu_0}(\Om)$, and let $\mu_\lambda' := \exp(f_0 + \lambda (f_1 - f_0)) \mu_0$ for $\lambda \in [0,1]$ be the segment joining $\mu_0'$ and $\mu_1'$. Then the following hold.
\begin{enumerate}
\item The measures $\mu'_\lambda$ are similar for $\lambda \in (0,1)$.
\item $\mu_\lambda' \preceq \mu_0'$ and $\mu_\lambda' \preceq \mu_1'$ for $\lambda \in (0,1)$.
\item $T_{\mu_\lambda'}\Mm_+(\Om, \mu_0) = T_{\mu_0'} \Mm_+(\Om, \mu_0) + T_{\mu_1'} \Mm_+(\Om, \mu_0)$ for $\lambda \in (0,1)$.
\end{enumerate}
\end{proposition}

\begin{proof}
Let $\delta := f_1 - f_0$ and $\phi := \exp(\delta)$. Then for all $\lambda_1, \lambda_2 \in [0,1]$, we have
\begin{equation} \label{eq:define-phi}
\mu'_{\lambda_1} = \phi^{\lambda_1-\lambda_2} \mu'_{\lambda_2}.
\end{equation}
For $\lambda_1 \in (0,1)$ and $\lambda_2 \in [0,1]$, we pick $p > 1$ such that $\lambda_2 + p (\lambda_1 - \lambda_2) \in (0,1)$. Then by (\ref{eq:define-phi}) we have
\[
\phi^{p (\lambda_1 - \lambda_2)} \mu'_{\lambda_2} = \mu'_{\lambda_2 + p (\lambda_1 - \lambda_2)}  \in \Mm_+(\Om, \mu_0),
\]
so that $\phi^{p (\lambda_1 - \lambda_2)} \in L^1(\Om, \mu'_{\lambda_2})$ or $\phi^{\lambda_1 - \lambda_2} \in L^p(\Om, \mu_{\lambda_2})$ for small $p-1 > 0$. Therefore, $\mu'_{\lambda_1} \preceq \mu'_{\lambda_2}$ for all $\lambda_1 \in (0,1)$ and $\lambda_2 \in [0,1]$, which implies the first and second statement.

This implies that $T_{\mu_i'} \Mm_+ (\Om, \mu_0) \subset T_{\mu'_\lambda} \Mm_+ (\Om, \mu_0) = T_{\mu'_{1/2}} \Mm_+ (\Om, \mu_0)$ for $i = 0, 1$ and all $\lambda \in (0,1)$ which shows one inclusion in the third statement. 

In order to complete the proof, observe that
\[
T_{\mu'_{1/2}} \Mm_+ (\Om, \mu_0) = T_{\mu'_{1/2}} \Mm_+ (\Om_+, \mu_0) \oplus T_{\mu'_{1/2}} \Mm_+ (\Om_-, \mu_0),
\]
where $\Om_+ := \{ \om \in \Om \mid \delta(\om) > 0\}$ and 
$\Om_- := \{ \om \in \Om \mid \delta(\om) \leq 0\}$. If $g \in T_{\mu'_{1/2}} \Mm_+ (\Om_+, \mu_0)$, then for some $t \neq 0$
\begin{eqnarray*}
\int_\Om \exp(|t g|) d\mu'_0 & \leq & \int_{\Om_+} \exp(|tg| + \frac12 \delta) d\mu'_0 + \int_{\Om_-} d\mu'_0\\
& = & \int_{\Om_+} \exp(|tg|) d\mu'_{1/2} + \int_{\Om_-} d\mu'_0 < \infty,
\end{eqnarray*}
so that $g \in T_{\mu'_0}(\Om, \mu_0)$ and hence, $T_{\mu'_{1/2}} \Mm_+ (\Om_+, \mu_0) \subset T_{\mu'_0}(\Om, \mu_0)$. Analogously, one shows that $T_{\mu'_{1/2}} \Mm_+ (\Om_-, \mu_0) \subset T_{\mu'_0}(\Om, \mu_1)$ which completes the proof.
\end{proof}

\section*{Acknowledgements} 
H.V.L. would  like to thank Shun-ichi Amari for many fruitful discussions, and Giovanni Pistone for providing the articles \cite{CP2007, GP1998}.
We thank Holger Bernigau for his critical helpful comments on an early version of this paper. We are  grateful to   the anonymous referees for their  helpful  remarks and suggestions.
This work has been supported by the Max-Planck Institute for Mathematics in the Sciences in Leipzig, the BSI at RIKEN  in Tokyo, the ASSMS, GCU in Lahore-Pakistan, the VNU for Sciences in Hanoi, the Mathematical  Institute of the Academy of Sciences of the 
Czech Republic in Prague, and the Santa Fe Institute. 
We are grateful for excellent working conditions and financial support of these institutions during extended visits of some of us.

\end{document}